\documentclass[12pt]{article}
\usepackage{amssymb,amsmath, amsfonts, amsthm}
\usepackage{bbding}
\usepackage[latin1]{inputenc}
\usepackage{pstricks,pst-node,pst-text,pst-3d,pst-plot,float}
\usepackage{enumerate}
\usepackage[total={6.5in,9.75in}, top=1.2in, left=0.9in, includefoot]{geometry}

\allowdisplaybreaks 

\newcommand{\norm}[1]{\left\Vert#1\right\Vert}

\newcommand{\abs}[1]{\left\vert#1\right\vert}
\newcommand{\Y}{\mathcal{Y}}
\newcommand{\X}{\mathcal{X}}
\newcommand{\D}{\mathcal{D}}

\newcommand{\R}{\mathcal{R}}
\newcommand{\Z}{\mathcal{Z}}
\newcommand{\RE}{{\rm I}\!{\rm R}} 

\newcommand{\ce}{{\scriptscriptstyle 0}}
\newcommand{\ds}{\displaystyle}
\newcommand{\s}{\scriptscriptstyle}

 \newtheorem{thm}{Theorem}[section]
 \newtheorem{cor}[thm]{Corollary}
 \newtheorem{lem}[thm]{Lemma}
 \newtheorem{prop}[thm]{Proposition}
 \newtheorem{defn}[thm]{Definition}
 \newtheorem{rem}[thm]{Remark}

\title{Regularization methods for ill-posed problems in multiple Hilbert scales \thanks{This work was
supported in part by Consejo Nacional de Investigaciones
Cient\'{\i}ficas y T\'{e}cnicas, CONICET, through PIP 2010-2012 Nro. 0219,
by Universidad Nacional del Litoral, U.N.L., through project CAI+D
2009-PI-62-315, by Agencia Nacional de Promoci\'{o}n Cient\'{\i}fica y
Tecnol\'{o}gia ANPCyT, through project PICT 2008-1301 and by the Air
Force Office of Scientific Research, AFOSR, through Grant
FA9550-10-1-0018.}}

\author{Gisela L. Mazzieri\thanks{Instituto de Matem\'{a}tica Aplicada del Litoral, IMAL,
CONICET-UNL, G\"{u}emes 3450, S3000GLN, Santa Fe, Argentina, and Departamento de Matem\'{a}tica,
 Facultad de Bioquimica y Ciencias Biologicas, Universidad Nacional del Litoral, Santa Fe,
 Argentina({\tt gmazzieri@hotmail.com}).} \and Ruben D. Spies$^\text{\Envelope,\,}$\thanks{Instituto de Matem\'{a}tica Aplicada
del Litoral, IMAL, CONICET-UNL, G\"{u}emes 3450, S3000GLN, Santa Fe,
Argentina and Departamento de Matem\'{a}tica, Facultad de Ingenier\'{\i}a
Qu\'{\i}mica, Universidad Nacional del Litoral, Santa Fe, Argentina
(\Envelope\,: {\tt rspies@santafe-conicet.gov.ar}).}}
\begin{document}
\date{\empty}
\maketitle
\begin{abstract}
Several convergence results in Hilbert scales under different
source conditions are proved and orders of convergence and optimal
orders of convergence are derived. Also, relations between those
source conditions are proved. The concept of a multiple Hilbert
scale on a product space is introduced, regularization methods on
these scales are defined, both for the case of a single
observation and for the case of multiple observations. In the
latter case, it is shown how vector-valued regularization
functions in these multiple Hilbert scales can be used. In all
cases convergence is proved and orders and optimal orders of
convergence are shown.
\end{abstract}



\pagestyle{myheadings} \thispagestyle{plain} \markboth{G. L.
MAZZIERI and R. D. SPIES.}{Regularization methods in multiple
Hilbert scales}
\setcounter{page}{1}

\section{Introduction}\label{intro}
Quite often an inverse problem can be formulated as the need for
determining $x$ in an equation of the form
\begin{equation}\label{eq:prob-inv}
Tx=y,
\end{equation}
where $T$ is a linear bounded operator between two infinite
dimensional Hilbert spaces $\X$ and $\Y$, the range of $T$,
$\R(T)$, is non-closed and $y$ is the data, which is known,
perhaps with a certain degree of error. It is well known that
under these hypotheses, problem (\ref{eq:prob-inv}) is ill-posed
in the sense of Hadamard (\cite{ref:Hadamard-1902}). The
ill-posedness is reflected in the fact that $T^\dag$, the
Moore-Penrose generalized inverse of $T$, is unbounded and
therefore small errors or noise in the data $y$ can result in
arbitrarily large errors in the corresponding approximated
solutions (see \cite{ref:Spies-Temperini-2006},
\cite{ref:Seidman-1980}), turning unstable all standard numerical
approximation methods, making them unsuitable for most
applications and inappropriate from any practical point of view.
The so called ``regularization methods" are mathematical tools
designed to restore stability to the inversion process and consist
essentially of parametric families of continuous linear operators
approximating $T^\dag.$ The mathematical theory of regularization
methods is very wide (a comprehensive treatise on the subject can
be found in the book by Engl, Hanke and Neubauer,
\cite{refb:Engl-Hanke-Neubauer-1996}) and it is of great interest
in a broad variety of applications in many areas such as Medicine,
Physics, Geology, Geophysics, Biology, image restoration and
processing, etc.

There exist numerous ways of regularizing an ill-posed inverse
problem. Among the most standar and traditional methods we mention
the Tikhonov-Phillips method (\cite{ref:Phillips-1962},
\cite{ref:Tikhonov-1963-SMD-2}, \cite{ref:Tikhonov-1963-SMD-1}),
truncated singular value decomposition (TSVD), Showalter's method,
total variation regularization  (\cite{ref:Acar-Vogel-1994}), etc.
Among all regularization methods, probably the best known and most
commonly and widely used is the Tikhonov-Phillips method, which
was originally proposed by Tikhonov and Phillips in 1962 and 1963
(see \cite{ref:Phillips-1962}, \cite{ref:Tikhonov-1963-SMD-2},
\cite{ref:Tikhonov-1963-SMD-1}). Although this method can be
formalized within a very general framework by means of spectral
theory (\cite{refb:Engl-Hanke-Neubauer-1996},
\cite{refb:Dautray-Lions-1990}), the widespread of its use is
undoubtedly due to the fact that it can also be formulated in a
very simple way as an optimization problem. In fact, the
regularized solution of problem (\ref{eq:prob-inv}) obtained by
applying Tikhonov-Phillips method is also the minimizer $x_\alpha$
of the functional
\begin{equation}\label{eq:formulacion variacional TP0}
J_\alpha(x)\doteq \norm{Tx-y}^2+\alpha\norm x^2,
\end{equation}
where $\alpha$ is a positive constant known as the regularization
parameter. The penalizing term $\alpha\norm x^2$ in
(\ref{eq:formulacion variacional TP0}) not only induces stability
but it also determines certain regularity properties of the
approximating regularized solutions $x_\alpha$ and of the
corresponding least-squares solution which they approximate as
$\alpha\to 0^+$. Thus, for instance, it is well known that
minimizers of (\ref{eq:formulacion variacional TP0}) are always
``smooth" and, for $\alpha\rightarrow 0^+$, they approximate the
least-squares solution of minimum norm of (\ref{eq:prob-inv}),
that is $\lim_{\alpha\to 0^+} x_\alpha=T^\dag y$. This method is
more precisely known as the Tikhonov-Phillips method of order
zero. Other penalizers in (\ref{eq:formulacion variacional TP0})
can also be used. For instance, in his original articles
(\cite{ref:Tikhonov-1963-SMD-2}, \cite{ref:Tikhonov-1963-SMD-1}),
Tikhonov considered the more general functional
\begin{equation}\label{eq:Tikhonov-general}
J_{\alpha(x),L}\doteq \norm{Tx-y}^2+\alpha\norm{Lx}^2,
\end{equation}
where $L$ is an operator defined on a certain domain
$\D(L)\subset\X$, into a Hilbert space $\Z$. Usually $L$ is a
differential operator and hence it has a nontrivial nullspace. In
spline smoothing problems for instance (see
\cite{refb:Wahba-1990}), $L$ is taken as the second derivative
operator.

The use of (\ref{eq:Tikhonov-general}) to regularize problem
(\ref{eq:prob-inv}) automatically implies the a-priori knowledge
or assumption that the exact solution belongs to $\D(L)$. This
approach gives rise to the theory of generalized inverses and
regularization with seminorms (see for instance
\cite{refb:Engl-Hanke-Neubauer-1996}, Chapter 8). The use of
Hilbert scales becomes appropriate when there is no certainty that
the exact solution is in fact an element of $\D(L)$.

The structure of this article is as follows. In Section 2 we
briefly recall the theory of regularization methods in Hilbert
scales. In Section 3 we prove several convergence results in
Hilbert scales under different source conditions and establish
orders of convergence and optimal orders of convergence. Also,
relations between those source conditions are proved. In Section 4
the concept of a multiple Hilbert scale on a product space is
introduced, regularization methods on these scales are defined,
first for the case of a single observation and then for the case
of multiple observations. In the latter case, it is shown how
vector-valued regularization functions in these multiple Hilbert
scales can be used. In all cases convergence is proved and orders
and optimal orders of convergence are shown.


\section{Regularization in Hilbert Scales}\label{sec:1}

In this section we will introduce the definition of a Hilbert
scale and a few known results that will be needed later. All of
them can be found in classical books and articles on the subject
such as \cite{refb:Engl-Hanke-Neubauer-1996} and
\cite{ref:Natterer-1984-AA}.

Throughout this work we will assume that $L$ is a densely defined,
unbounded, strictly positive self-adjoint operator on a Hilbert
space $\X$, so that $L$ is closed and satisfies $\langle
Lx,y\rangle = \langle x, Ly\rangle$ for every $x,y\in\D(L)$ and
there exists a positive constant $\gamma$ such that
\begin{equation}\label{eq:constante de L estric posit}
\langle Lx,x\rangle\geq \gamma \norm{x}^2 \hspace{0.2cm} \text{for
every } x\in\D(L).
\end{equation}
Consider the set $\mathcal M$ of all elements $x\in\X$ for which
all natural powers of $L$ are defined, that is $\displaystyle
\mathcal M\doteq \bigcap _{k=1} ^\infty \D(L^k)$. By using
spectral theory it can be easily shown that the fractional powers
$L^s$ are well defined over $\mathcal M$ for every $s\in\mathbb R$
and that
\begin{equation} \label{def:the-set-M}
\mathcal M = \bigcap
_{s\in\mathbb R} \D(L^s)
\end{equation}
(for a detailed and comprehensive treatment of fractional powers
of strictly positive self-adjoint operators see for instance
\cite{refb:Pazy-1983} and \cite{refb:Dautray-Lions-1990}).
\begin{defn}\label{def:escalas de Hilbert}
(Hilbert scales) Let $\mathcal M$ be defined as in
(\ref{def:the-set-M}). For every $t\in\mathbb R$ we define
\begin{equation} \label{def:s-inner-product}
\langle x, y\rangle_t \doteq \langle L^tx, L^ty\rangle,\quad
\text{for }x,y,\in\mathcal M.
\end{equation}
It can be immediately seen that $\langle\cdot,\cdot\rangle_t$
defines an inner product in $\mathcal M$, which in turn induces a
norm $\|x\|_t=\|L^tx\|$. The Hilbert space $\X_t$ is defined as
the completion of $\mathcal M$ with respect to this norm
$\norm{\cdot}_t$. The family of spaces $(\X_t)_{t\in\mathbb R}$ is
called the Hilbert scale induced by $L$ over $\X$. The operator
$L$ is called a ``generator" of the Hilbert scale
$(\X_t)_{t\in\mathbb R}$.
\end{defn}
\begin{rem}
Note that a Hilbert scale is a completely ordered (by set inclusion)
parametric family of Hilbert spaces and if the operator $L$ is
bounded then $\X_t=\X$ for every $t\in\mathbb R$.
\end{rem}
The following proposition constitutes one of the fundamental
results for the treatment of inverse ill-posed problems in Hilbert
scales.

\begin{prop}\label{prop:8.19}
Let $(\X_t)_{t\in\mathbb R}$ be the Hilbert scale induced  by $L$
over $\X$. Then the following is true:
\begin{enumerate}[i)]
\item For every  $s,t\in\mathbb R$  such that $-\infty<s<t<\infty$, the space $\X_t$
is continuously and densely embedded in $\X_s$.
\item Let $s,t\in\mathbb R$. The operator $L^{t-s}$ defined on $\mathcal
M$ has a unique extension to $\X_t$ which is an isomorphism
(surjective isometry) from $\X_t$ onto $\X_s$. This extension,
also denoted with $L^{t-s}$, is self-adjoint and strictly positive
seen as an operator in $\X_s$ with domain $\X_t$, if $t>s$. Also,
the identity $L^{t-s}=L^tL^{-s}$ is valid for the appropriate
extensions. In particular $(L^s)^{-1}=L^{-s}$.
\item If $s\geq 0$, then $\X_s=\D(L^s)$ and $\X_{-s}=(\X_s)'$;
that is $\X_{-s}$ is the topological dual of $\X_s$ (with the
topology induced by the norm in $\X$).
\item Let $q,r,s\in\mathbb R$ be such that $-\infty<q<r<s<\infty$ y $x\in\X_s$. Then the
following interpolation inequality holds:
\begin{equation}\label{eq:des-interp-Esc Hilb}
\norm x_r\leq
\norm{x}_q^{\frac{s-r}{s-q}}\norm{x}_s^{\frac{r-q}{s-q}}.
\end{equation}
\end{enumerate}
\end{prop}
\begin{proof}
See \cite{refb:Engl-Hanke-Neubauer-1996}, Proposition 8.19.
\end{proof}
In the remaining of this section we will state several results
which will be of fundamental importance in the following sections.
In all cases we have been included appropriate references where
their proofs can be found.
\begin{thm}\label{teo:desigualdad de Heinz}
(Heinz Inequality) Let $A$ and $L$ be two linear, unbounded
densely defined, strictly positive, self-adjoint operators on a
Hilbert Space $\X$ such that
\begin{equation}\label{eq:inclusion dominios Heinz}
\D(A)\subset\D(L)
\end{equation}
and
\begin{equation}\label{eq:1 teo Heinz}
\norm{Lx}\leq\norm{Ax} \hspace{0.5cm}\forall \,\,x\in\D(A).
\end{equation}
Then for every $\nu\in [0,1]$ there holds
\begin{equation}\label{eq:2 teo Heinz}
\D(A^\nu)\subset\D(L^\nu)
\end{equation}
 and
\begin{equation}\label{eq:3 teo Heinz}
\norm{L^\nu x}\leq\norm{A^\nu x} \hspace{0.5cm} \forall \,\,x\in
\D(A^\nu).
\end{equation}
\end{thm}
\begin{proof}
See \cite{refb:Engl-Hanke-Neubauer-1996}, Proposition 8.21, page
213 (see also \cite{ref:Heinz-1951} and
\cite{ref:Krein-Petunin-1966}).
\end{proof}
\begin{rem}
It is important to point out here that the result of Theorem
\ref{teo:desigualdad de Heinz} remains valid under slightly weaker
hypotheses on the involved operators. More precisely, it can be
shown that the result remains valid if the operators $A$ and $L$
satisfy conditions (\ref{eq:inclusion dominios Heinz}) and
(\ref{eq:1 teo Heinz}) and are self-adjoint and nonnegative
instead of strictly positive.
\end{rem}
\begin{lem}\label{lem:R(T*)=Xa}
Let $T:\X\longrightarrow\Y$ be a linear bounded operator between
the Hilbert spaces $\X$ and $\Y$ and $L$ a linear, densely
defined, self-adjoint, unbounded and strictly positive operator on
the space $\X$. Let $(\X_t)_{t\in\mathbb R}$ be the Hilbert scale
induced by $L$ over $\X$. If there exist constants $0<m\leq
M<\infty$ and $a\in\mathbb R^+$ such that
\begin{equation}\label{eq:LT suavidad}
m\norm{x}_{-a}\leq\norm{Tx}\leq M \norm{x}_{-a}\hspace{0.3cm}\forall\,\, x\in\X,
\end{equation} then $\R(T^*)= \X_a$ (that is,
$\R(T^*)=\D(L^a)=\R(L^{-a}$)).
\end{lem}
\begin{proof}
See \cite{refthes:Egger-2005}.
\end{proof}
\begin{rem}\label{obs:relacion T y L suavidad}
Note that if (\ref{eq:LT suavidad}) holds, then the operator $T$
is injective. Also note that (\ref{eq:LT suavidad}) essentially
says that the operator $T$ induces a norm on $\X$ which is
equivalent to that inherited by $\X$ from the Hilbert scale of
order $t=-a$, generated by the operator $L$ over $\X$. Hence, it is
reasonable to think, in intuitive terms, that the degree of
regularity induced by $T$ is equivalent to the degree of
regularity induced by $L^{-a}$, and therefore the same happens
with the degree of ill-posedness of their respective inverses.
\end{rem}
\begin{thm}\label{theorem:cota B*B}
Let $T:\X\longrightarrow\Y$ be a linear bounded operator between
the Hilbert spaces $\X$ and $\Y$  and $L$ a linear, densely
defined, self-adjoint, unbounded and strictly positive operator on
$\X$. Let $(\X_t)_{t\in\mathbb R}$ be the Hilbert scale induced by
the operator $L$ over $\X$. Suppose that the operator $T$
satisfies (\ref{eq:LT suavidad}) for some $a>0$ and $0<m\leq
M<\infty$. Given $s>0$ define $B\doteq TL^{-s}$ where $L^{-s}$ is
considered extended to all $\X$ in the sense of Proposition
\ref{prop:8.19} ii). Then for every $\nu\in[0,1]$ we have that
\begin{eqnarray}\label{eq:cota B*B}
m^\nu\norm{x}_{-\nu(a+s)}&\leq \norm{(B^*B)^{\frac{\nu}{2}}x}&\leq
M^\nu \norm{x}_{-\nu(a+s)},\hspace{0.2cm}\forall \,\,x\in\X,\\
M^{-\nu}\norm{x}_{\nu(a+s)}&\leq
\norm{(B^*B)^{-\frac{\nu}{2}}x}&\leq
m^{-\nu}\norm{x}_{\nu(a+s)},\hspace{0.2cm}\forall
\,\,x\in\D((B^*B)^{-\frac{\nu}{2}}).
\end{eqnarray}
Also
\begin{equation}\label{eq:igualdad de rango BestrellaB en teorema}
\R\left((B^*B)^{\frac{\nu}{2}}\right)=\X_{\nu(a+s)}.
\end{equation}
\end{thm}
\begin{proof}
See \cite{refthes:Egger-2005} (see also
\cite{refb:Engl-Hanke-Neubauer-1996} Corollary 8.22, page 214).
\end{proof}
\begin{rem}
If the operators  $L^{-1}$ y $T^*T$ commute, then
(\ref{eq:igualdad de rango BestrellaB en teorema}) remains valid
also for $\nu>1$. This result, which we will prove later on
(Theorem \ref{theorem:extension theorem B*B}), will be of
fundamental importance in the extension of some results on
convergence of some regularization methods in Hilbert scales,
which will be presented in Section 3.
\end{rem}
The inequalities in (\ref{eq:cota B*B}) can be interpreted in a
similar way as it was done for (\ref{eq:LT suavidad}) in the
Remark \ref{obs:relacion T y L suavidad}. In fact, taking as
``unit of regularity" the degree induced by the operator $L^{-1}$,
the respective degrees of regularity induced by $L^{-s}$ and $T$
are $s$ and $a$, respectively. Hence the degree induced by
$B=TL^{-s}$ is $a+s$, the degree induced by $B^*B$ is $2(a+s)$
and, therefore, the degree of regularity induced by
$(B^*B)^{\frac{\nu}{2}}$ is $\frac{\nu}{2}2(a+s)=\nu(a+s)$.

The idea of using Hilbert scales for regularizing inverse
ill-posed problems was first introduced by Natterer in 1984
(\cite{ref:Natterer-1984-AA}) for the special case of the
classical Tikhonov-Phillips method. In his work Natterer
regularized the problem $Tx=y$ by minimizing the functional
\begin{equation}\label{eq:Tik-EH}
\norm{Tx-y^{\scriptscriptstyle \delta}}^2+\alpha\norm{x}_{\scriptscriptstyle s}^2,
\end{equation}
over the space $\X_{\scriptscriptstyle s}$, where
$\norm{\cdot}_{\scriptscriptstyle s}$ denotes the corresponding norm in
the Hilbert scale (see Definition \ref{def:escalas de Hilbert}).

In certain cases it is possible that a value of $s_{\ce}>0$ be
known for which we are absolutely sure that the exact solution
$x_{\ce}^\dag\in\X_{s_{\ce}}$, where $(\X_t)_{t\in\mathbb R}$ is
the Hilbert scale induced by the operator $L$ over $\X$. In such
cases it is possible to proceed with regularization of the problem
$Tx=y$ by means of the traditional methods, by replacing the
Hilbert space $\X$ by $\X_{s_{\ce}}$ and, obviously $T$ by its
restriction to $\D(L^{s_{\ce}})$. In other cases, however, it is
possible that such a value of $s_\ce$  be not exactly known,
although it could be reasonable to assume the existence of some
$u>0$ for which
\begin{equation}\label{eq:cond fuente xmas}
x^\dag\in\X_u,
\end{equation}
(although the exact value of $u$ be unknown). It is precisely in
this case in which Hilbert scales provide a solid mathematical
framework for the development of convergent regularization methods
which allow us to take advantage, in a optimal and ``adaptive"
way, of the source condition (\ref{eq:cond fuente xmas}) in order
to obtain the best possible convergence speed, even though $u$ is
unknown.

The first result about convergence on Hilbert scales is due to F.
Natterer (\cite{ref:Natterer-1984-AA}) and is presented in the
next theorem.
\begin{thm}\label{teo:ordenes de convergencia}
Let $T\in\mathcal L(\X,\Y)$ with $\X$ and $\Y$ Hilbert spaces,
$T^\dag$ the Moore-Penrose generalized inverse of $T$,
$L:\D(L)\subset\X\longrightarrow\X$ a linear, densely defined,
self-adjoint, unbounded operator with $L\geq\gamma$ for some
$\gamma>0$ and $(\X_t)_{t\in\mathbb R}$ the Hilbert scale induced
by $L$ over $\X$. Suppose also that condition (\ref{eq:LT
suavidad}) holds. Let $s\geq 0$ and $B= TL^{-s}$, as in Theorem
\ref{theorem:cota B*B}. Let $g_\alpha:[0,\norm{B}^2]\rightarrow
\mathbb R,\, \alpha>0,$ be a family of piecewise continuous
functions and  $r_\alpha(\lambda)\doteq 1-\lambda
g_\alpha(\lambda)$. Suppose also that $\{g_\alpha\}$ satisfies the
following conditions:
\begin{eqnarray}
\text{C1}&:&\forall \,\,\lambda \in(0,\norm{B}^2] \text{ we have that }\lim_{\alpha\rightarrow 0^{\scriptscriptstyle
+}}g_\alpha(\lambda)=\frac{1}{\lambda};\label{eq:limgalpha}\\
\text{C2}&: & \exists \,\,\hat c>0 \text{ such that }\forall \,\,
\lambda \in(0,\norm{B}^2] \text{ and } \forall\,\,\alpha>0 \text{
there holds } \abs{g_\alpha(\lambda)}\leq \hat
c\alpha^{-1};\label{eq:cotagalpha}\\
\text{C3}&: & \exists \,\,\mu_{\scriptscriptstyle 0}\geq 1 \text{
such that if }\mu\in[0,\mu_{\scriptscriptstyle 0}]\text{ then
}\lambda^\mu\abs{r_\alpha(\lambda)} \leq
c_\mu\alpha^\mu\hspace{0.2cm} \forall \,\,\lambda
\in(0,\norm{B}^2],\label{eq:calif}
\end{eqnarray}
where $c_\mu$ is a positive constant.

For $y\in\D(T^\dag)$, $y^\delta\in\Y$ with $\norm{y-y^\delta}\leq
\delta$ we define the regularized solution of the problem
$Tx=y^\delta$ by
\begin{equation}\label{eq:reg esc Hilbert}
x_\alpha ^\delta\doteq R_\alpha y^\delta \doteq L^{-s}g_\alpha(B^*B)B^*y^\delta.
\end{equation}
 Suppose that $x^\dag=T^\dag y \in \X_u$ for some $u\in [0,a+2s]$
 and that the regularization parameter $\alpha$ is chosen as
\begin{equation}\label{eq:elecalpha}
\alpha\doteq c\left(
\frac{\delta}{\norm{x^\dag}_u}\right)^{\frac{2(a+s)}{a+u}},
\end{equation}
where $c$ is a positive constant and  $a$ is the constant in
(\ref{eq:LT suavidad}). Then there exists a constant $C$ (which
depends on $a$ and $s$ but not on $u$) such that the following
estimate for the total error holds:
\begin{equation}\label{eq:orden de conv theorem 8.23}
\norm{x_\alpha^\delta-x^\dag}\leq
C\,\norm{x^\dag}_u^{\frac{u}{a+u}} \delta^{\frac{u}{a+u}}.
\end{equation}
\end{thm}
\begin{proof}
See \cite{refb:Engl-Hanke-Neubauer-1996}, Theorem 8.23.
\end{proof}

In Figure \ref{fig:elecciones de s} the relation among the values
of the parameters $s$ and $u$ of Theorem \ref{teo:ordenes de
convergencia} is schematized. Observe that the largest possible
value for $s$ is $\frac{u-a}{2}$. The arrow indicates the space
$\X_s$ may or may not be contained in $\X_u$. The dashed curve
represents the space $\X_u$ indicating that the parameter $u$ is
unknown. \vspace{0.2cm}

\begin{figure}[H]
\begin{pspicture}(-6.5,-2)(3,2)
\psellipse(0,0)(3,2)
\psellipse[linestyle=dashed,dotsep=1pt,linewidth=0.5pt](0,0)(2.1,1.6)
\rput[t](1.4,1.6){$\X_u$}
\rput[t](2,2){$\X_{\frac{u-a}{2}}$}
\psline{<->}(1.3,0)(3,0)
\rput[t](2.5,0.4){$\X_s$}
\end{pspicture}
\caption{The Hilbert scales in Theorem \ref{teo:ordenes de
convergencia}.}\label{fig:elecciones de s}
\end{figure}
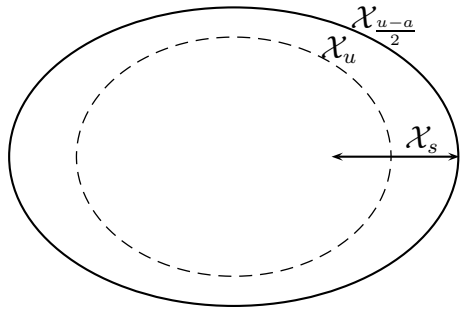

\begin{rem}
It is very important to point out the ``adaptivity" of the order
of convergence in Theorem \ref{teo:ordenes de convergencia}. In
fact, note that although the regularized solutions
$x_\alpha^\delta$ defined in (\ref{eq:reg esc Hilbert}) do not
depend on the degree of regularity $u$ of $x^\dag$, the order of
convergence obtained does depend on $u$. This order improves as
$u$ increases and it becomes asymptotically optimal in $u$. Also
observe that in order to assure the order of convergence in
(\ref{eq:orden de conv theorem 8.23}) it is necessary to choose
$s$ (note that $R_\alpha$ depends on $s$) such that $u\leq a+2s$.
Since it is possible that $u$ be unknown, it may happen that we
may not be completely sure of the validity of such constraint.
Note that in such a case, i.e. if $u>a+2s$, an order of
convergence $O\left(\delta^{\frac{u}{a+u}}\right)$ cannot be
guaranteed for the total error. However, since
$\X_u\subset\X_\eta\hspace{0.2cm}\forall \,u\geq\eta$, in such
circumstances we will still obtain at least convergence of the
order
$O\left(\delta^{\frac{a+2su}{a+(a+2s)}}\right)=O\left(\delta^{\frac{a+2s}{2(a+s)}}\right)$.
Thus, not choosing $s$ sufficiently large will result in a worse
order of convergence.
\end{rem}

\section{Preliminary convergence results in Hilbert scales}\label{sec:2}
In the next theorem, which extends the results of Theorem
\ref{teo:ordenes de convergencia}, we will show that convergence
can be obtained when the parameter choice rule $\alpha$ is chosen
in the form $\alpha=c\,\delta^\varepsilon$  for all values of
$\varepsilon$ in a certain interval, and not only for
$\varepsilon=\frac{2(a+s)}{a+u}$, corresponding to the choice in
(\ref{eq:elecalpha}). We will prove however that for this choice
of $\varepsilon$ the order of convergence is optimal.

\begin{thm}\label{coro:coro de ordenes de convergencia}
Let $\X$, $\Y$, $T$, $T^\dag$, $L$, $(\X_t)_{t\in\mathbb R}$,
$s\geq 0$, $a>0$, $B=TL^{-s}$, $g_\alpha,$ $r_\alpha$,
$R_\alpha=L^{-s}g_\alpha(B^*B)B^*$, $y\in\D(T^\dag)$, $y^\delta
\in\Y$, $\|y-y^\delta\|\le\delta$, $u\in[0,a+2s],$ $x^\dag=T^\dag
y\in\X_u$, $x_\alpha=R_\alpha y$ y $x_\alpha^\delta=R_\alpha
y^\delta,$ all as in Theorem \ref{teo:ordenes de convergencia}. If
the parameter choice rule $\alpha$ is chosen as
$\alpha=c\,\delta^\varepsilon$ ($c$ constant) and $\varepsilon
\in\left(0,\frac{2(a+s)}{a}\right)$ then:
\begin{enumerate}
\item [i)] As $\delta\rightarrow 0^{\s+}$, $x_\alpha^\delta\rightarrow x^\dag$ in $\X$.
\item [ii)] Moreover, $\norm{x_\alpha^\delta-x^\dag}=\mathcal
O(\delta^\sigma)$ where
$\sigma=\min\left\{1-\frac{a\varepsilon}{2(a+s)},\frac{u\varepsilon}{2(a+s)}\right\}>0$.
\item [iii)] The order of convergence for the total error is optimal when $\varepsilon$
is chosen as $\varepsilon=\frac{2(a+s)}{a+u},$ in which case
$\norm{x_\alpha^\delta-x^\dag}=\mathcal
O(\delta^{\frac{u}{a+u}})$.
\end{enumerate}
\end{thm}
\begin{proof}
First note that from conditions (\ref{eq:cotagalpha}) and
(\ref{eq:calif}) it follows immediately that there exists a
constant $k>0$ such that
\begin{equation}\label{eq:cota galpha}
\lambda^\beta |g_\alpha(\lambda)|\leq
k\alpha^{\beta-1},\hspace{0.5cm} \forall\,\beta\in[0,1],\,
\forall\, \alpha>0\, \text{ and } \,\forall\, \lambda \in
(0,\norm{B}^2],
\end{equation}
(we can take $k=\text{max}\{1+c_{\s 0},\hat{c}\}$ where $c_{\s 0}$
is the constant $c_\mu$ in (\ref{eq:calif}) corresponding to
$\mu=0$).

We will now proceed to estimate the error due to noise in the data
and the regularization error, separately. Without loss of
generality we will suppose that $y\in\R(T)$ (otherwise we replace
$y$ by $Q\,y$ where
$Q:\Y\overset{\perp}{\longrightarrow}\overline{\R(T)}$; recall
that $y\in\D(T^\dag)$ and $T^\dag y=T^\dag\,Q\,y$).

For the error due to noise we have:
\begin{align*}
\norm{x_\alpha^\delta -x_\alpha}&= \norm{R_\alpha(y^\delta -y)}\\
&= \norm{L^{-s}g_\alpha(B^*B)B^*(y^\delta-y)} \\
&=\norm{g_\alpha(B^*B)B^*(y^\delta-y)}_{-s} \\
&\leq
m^{-\frac{s}{(a+s)}}\norm{(B^*B)^{\frac{s}{2(a+s)}}g_\alpha(B^*B)B^*(y^\delta-y)}\qquad\qquad
\left(\text{by }
(\ref{eq:cota B*B}) \text{ with } \nu\doteq\frac{s}{a+s}\right)\\
&= m^{-\frac{s}{(a+s)}}\norm{(B^*B)^{-1/2}(B^*B)^{\frac{a+2s}{2(a+s)}}g_\alpha(B^*B)B^*(y^\delta-y)}\\
&= m^{-\frac{s}{(a+s)}}\norm{(B^*B)^{\frac{a+2s}{2(a+s)}}g_\alpha(B^*B)(B^*B)^{-1/2}B^*(y^\delta-y)} \\
&\leq m^{-\frac{s}{(a+s)}}\, k
\alpha^{\frac{-a}{2(a+s)}}\norm{(B^*B)^{-1/2}B^*(y^\delta-y)}
\qquad\qquad
\left(\text{by (\ref{eq:cota galpha}) with }\beta\doteq\frac{a+2s}{2a+2s}\right)\\
&\leq
c_1\delta\alpha^{\frac{-a}{2(a+s)}},\qquad\qquad\qquad\qquad\qquad\qquad\qquad\left(\text{since
}\norm{(B^*B)^{-1/2}z}=\norm{(B^*)^{-1}z}\right)
\end{align*}
where $C_1=k\,m^{-\frac{s}{(a+s)}}$. Therefore
\begin{equation}\label{eq:cota ruido}
\norm{x_\alpha^\delta -x_\alpha}\leq C_1\delta\alpha^{\frac{-a}{2(a+s)}}.
\end{equation}
At this point it is timely to note that the estimate for the error
due to noise in (\ref{eq:cota ruido}) is independent of the degree
of regularity $u$ of the solution $x^\dag$.

Next we proceed to estimate the regularization error
$\norm{x_\alpha-x^\dag }$. Note in first place that from
Proposition \ref{prop:8.19} \textit{ii)} (with $t=u$ and $s=u-s$),
it follows that $L^{u-(u-s)}=L^s$ has a unique extension to $\X_u$
which is an isomorphism from $\X_u$ onto $\X_{u-s}$. It is
important to point out here that it is precisely this property of
the fractional powers of the operator $L$ on the Hilbert scales
induced by itself, what will allow us, in the end, to arrive to
the adaptive convergence order that we want to prove. More
precisely, note that whatever the value of $u$ (perhaps unknown),
$L^s$ always possesses a unique extension to $\X_u$. This
extension, also denoted with $L^s$, regarded as an operator on
$\X_{u-s}$ with domain $\X_u$, is self-adjoint and strictly
positive if $u> u-s$. Then, since $x^\dag\in\X_u$, it follows that
\begin{equation}\label{eq:6 teo 8.23}
L^sx^\dag\in\X_{u-s}.
\end{equation}
\noindent On the other hand, if $u\geq s$, from Theorem
$\ref{theorem:cota B*B}$ con $\nu\doteq\frac{u-s}{a+s}$ it follows
that
\begin{equation}\label{eq:7 teo 8.23}
\X_{u-s}=\R\left((B^*B)^{\frac{u-s}{2(a+s)}}\right).
\end{equation}
\noindent From (\ref{eq:6 teo 8.23}) and (\ref{eq:7 teo 8.23}) it
follows that there exist $v\in\X$ such that
\begin{equation}\label{eq:8 teo 8.23}
L^sx^\dag=(B^*B)^{\frac{u-s}{2(a+s)}}v.
\end{equation}
\noindent If  $u<s$ then (\ref{eq:8 teo 8.23}) holds with $v\doteq
(B^*B)^{\frac{s-u}{2(a+s)}}L^sx^\dag.$

Then,
\begin{align*}
\norm{x_\alpha-x^\dag}&= \norm{R_\alpha y-x^\dag}&\\
&= \norm{L^{-s}g_\alpha(B^*B)B^*y-x^\dag}&\\
&= \norm{L^{-s}g_\alpha(B^*B)B^*BL^s x^\dag-x^\dag}& (\text{since }B^*y=B^*BL^s x^\dag)\\
&= \norm{L^{-s}g_\alpha(B^*B)B^*BL^s x^\dag-L^{-s}L^sx^\dag}& \\
&= \norm{L^{-s}[g_\alpha(B^*B)B^*B-I]L^sx^\dag}&\\
&= \norm{L^{-s}r_\alpha(B^*B)L^sx^\dag}& \\
&=\norm{L^{-s}r_\alpha(B^*B)(B^*B)^{\frac{u-s}{2(a+s)}}v}& (\text{by (\ref{eq:8 teo 8.23})})\\
&=\norm{r_\alpha(B^*B)(B^*B)^{\frac{u-s}{2(a+s)}}v}_{-s}&  \\
&=\norm{(B^*B)^{\frac{u-s}{2(a+s)}}r_\alpha(B^*B)v}_{-s}&\\
&\leq m^{-\frac{s}{(a+s)}} \norm{(B^*B)^{\frac{s}{2(a+s)}}(B^*B)^{\frac{u-s}{2(a+s)}}r_\alpha(B^*B)v} &
\left(\text{by (\ref{eq:cota B*B}) with }\nu\doteq\frac{s}{a+s}\right)\\
&= m^{-\frac{s}{(a+s)}} \norm{(B^*B)^{\frac{u}{2(a+s)}}r_\alpha(B^*B)v}\\
&\leq m^{-\frac{s}{(a+s)}} \,c_{\bar{\mu}}\alpha^{\frac{u}{2(a+s)}}\norm{v}& \left(\text{by (\ref{eq:calif}) with }\bar{\mu}\doteq\frac{u}{2(a+s)}\right)\\
&= m^{-\frac{s}{(a+s)}}\,c_{\bar{\mu}}\alpha^{\frac{u}{2(a+s)}}\norm{(B^*B)^{\frac{s-u}{2(a+s)}}L^sx^\dag} &(\text{by (\ref{eq:8 teo 8.23})})\\
&\leq m^{-\frac{s}{(a+s)}}\,c_{\bar{\mu}}\,\alpha^{\frac{u}{2(a+s)}}\, M^{\frac{s-u}{a+s}}\norm{L^sx^\dag}_{u-s}& \left(\text{by (\ref{eq:cota B*B}) with }\, \nu\doteq\frac{s-u}{a+s}\right)\\
&= m^{-\frac{s}{(a+s)}}\,c_{\bar{\mu}}\,
M^{\frac{s-u}{a+s}}\alpha^{\frac{u}{2(a+s)}}\norm{x^\dag}_u.&
\end{align*}
Hence, there exists  $C_2\doteq
m^{-\frac{s}{(a+s)}}\,c_{\bar{\mu}}\, M^{\frac{s-u}{a+s}} $ such
that
\begin{equation}\label{eq:cota reg}
\norm{x_\alpha-x^\dag}\leq C_2\norm{x^\dag}_u \alpha^{\frac{u}{2(a+s)}}.
\end{equation}
Note that this estimate for the regularization error depends on
the degree of regularity $u$ of $x^\dag$ and it is relevant only
for the case $u>0$.

Finally, from (\ref{eq:cota ruido}) and (\ref{eq:cota reg}) it
follows that
\begin{align}
\norm{x_\alpha^\delta-x^\dag}&\leq\norm{x_\alpha^\delta-x_\alpha}+ \norm{x_\alpha-x^\dag} &\nonumber\\
&\leq C_1\delta\alpha^{\frac{-a}{2(a+s)}}+ C_2\alpha^{\frac{u}{2(a+s)}}\norm{x^\dag}_u & \hspace{-1cm}\nonumber\\
&= C_1\delta\left(c\delta^{\varepsilon}\right)^{\frac{-a}{2(a+s)}}
+C_2\left(c\delta^{\varepsilon}\right)^{\frac{u}{2(a+s)}}\norm{x^\dag}_u &\nonumber\\
&= C_1c^{\frac{-a}{2(a+s)}}\delta^{1-\frac{\varepsilon a}{2(a+s)}}
+C_2c^{\frac{u}{2(a+s)}}\norm{x^\dag}_u\delta^{\frac{\varepsilon u}{2(a+s)}} &\label{eq:1 coro ordenes}\\
&=\mathcal O(\delta^\sigma),\nonumber
\end{align}
where
$\sigma=\min\left\{1-\frac{a\varepsilon}{2(a+s)},\frac{u\varepsilon}{2(a+s)}\right\}$.
This proves \textit{i)} and \textit{ii)}.

To prove \textit{iii)}, note that by virtue of (\ref{eq:1 coro
ordenes}) it follows that the order of convergence is optimal when
$\varepsilon$ is chosen such that
$$1-\frac{a\varepsilon}{2(a+s)}= \frac{\varepsilon
u}{2(a+s)},$$
that is for $\varepsilon =\frac{2(a+s)}{a+u}$, in which case
$\sigma=\frac{u}{a+u}$. It is important to note here that this
optimal order of convergence depends on $a$ and  $u$ (that is on
$L,\,T$ and $x^\dag$) but it does not depend on the choice of $s$.
\end{proof}

In the next theorem we will prove that with the same parameter
choice rule as in (\ref{eq:elecalpha}), it is possible to obtain a
better order of convergence in a weaker norm or convergence in a
stronger norm with a worse order.
\begin{thm}\label{teo:convergencia normas intermedias}
Let $\X$, $\Y$, $T$, $T^\dag$, $L$, $(\X_t)_{t\in\mathbb R}$,
$s\geq 0$, $a>0$, $B=TL^{-s}$, $g_\alpha,$ $r_\alpha$,
$R_\alpha=L^{-s}g_\alpha(B^*B)B^*$, $y\in\D(T^\dag)$, $y^\delta
\in\Y$, $\|y-y^\delta\|\le\delta$, $u\in[0,a+2s],$ $x^\dag=T^\dag
y\in\X_u$, $x_\alpha=R_\alpha y$ y $x_\alpha^\delta=R_\alpha
y^\delta,$ all as in Theorem \ref{teo:ordenes de convergencia}.
Suppose that the parameter choice rule $\alpha$ is chosen as in
(\ref{eq:elecalpha}), that is
\begin{equation}\label{eq:elecalpha normas r}
\alpha=c\left(
\frac{\delta}{\norm{x^\dag}_u}\right)^{\frac{2(a+s)}{a+u}},
\end{equation}
where $c>0$. Then for every $r\in[-a,\min\{u,s\}]$ there holds
\begin{equation}\label{eq:orden de conv theorem normas r}
\norm{x_\alpha^\delta-x^\dag}_r\leq
C \, \norm{x^\dag}_u^{\frac{a+r}{a+u}}\delta^{\frac{u-r}{a+u}},
\end{equation}
where $C$ is a constant depending on $a$, $s$ and $r$ but not on
$u$ nor on $x^\dag$.
\end{thm}

\begin{proof}
First note that due to the restriction on $r$, we have that
$x^\dag,\,x_\alpha,\, x_\alpha^\delta$ are all in $\X_r$. Just
like in the previous theorem, without loss of generality we will
suppose that $y\in\R(T)$.

For the error due to noise we have the following estimate:
\begin{align*}
\norm{x_\alpha^\delta-x_\alpha}_r&=\norm{R_\alpha(y^\delta-y)}_r &\\
&=\norm{L^{-s}g_\alpha(B^*B)B^*(y^\delta-y)}_r&\\
&=\norm{g_\alpha(B^*B)B^*(y^\delta-y)}_{r-s}&\\
&\leq m^{\frac{r-s}{a+s}}\norm{(B^*B)^{\frac{s-r}{2(a+s)}}g_\alpha(B^*B)B^*(y^\delta-y)}& \hspace{-0.6cm}\left(\text{by (\ref{eq:cota B*B}) with }\, \nu=\frac{s-r}{a+s}\right)\\
&= m^{\frac{r-s}{a+s}} \norm{(B^*B)^{\frac{a+2s-r}{2(a+s)}}(B^*B)^{-1/2} g_\alpha(B^*B)B^*(y^\delta-y)}& \\
&\leq m^{\frac{r-s}{a+s}}\norm{(B^*B)^{\frac{a+2s-r}{2(a+s)}}g_\alpha(B^*B) (B^*B)^{-1/2}B^*(y^\delta-y)}& \\
&\leq  m^{\frac{r-s}{a+s}}\, k\alpha^{\frac{-r-a}{2(a+s)}}\norm{(B^*B)^{-1/2}B^*(y^\delta-y)}& \hspace{-0.6cm}\left(\text{by (\ref{eq:cota galpha})  with } \beta\doteq\frac{a+2s-r}{2(a+s)}\right)\\
&=m^{\frac{r-s}{a+s}}\, k\alpha^{\frac{-r-a}{2(a+s)}}\norm{y-y^\delta}&\\
&\leq  m^{\frac{r-s}{a+s}}\, k \alpha^{-\frac{a+r}{2(a+s)}}\delta&\\
&\leq m^{\frac{r-s}{a+s}}\, k\left[c\left(\frac{\delta}{\norm{x^\dag}_u}\right)^{\frac{2(a+s)}{a+u}}\right]^{-\frac{a+r}{2(a+s)}}\delta &\hspace{-0.6cm}(\text{by (\ref{eq:elecalpha normas r})})\\
&= C_1 \norm{x^\dag}_u^{\frac{a+r}{a+u}}\delta^{\frac{u-r}{a+u}},&
\end{align*}
where $C_1=m^{\frac{r-s}{a+s}}\, k\, c^{-\frac{a+r}{2(a+s)}}$.
Thus
\begin{equation}\label{eq:cota ruido norma r}
\norm{x_\alpha^\delta-x_\alpha}_r\leq C_1\norm{x^\dag}_u^{\frac{a+r}{a+u}}\delta^{\frac{u-r}{a+u}}.
\end{equation}
For the regularization error note that:
\begin{align*}
\norm{x_\alpha-x^\dag}_r &= \norm{R_\alpha y-x^\dag }_r&\\
&= \norm{L^{-s}g_\alpha(B^*B)B^*y-x^\dag}_r & \\
&= \norm{L^{-s}g_\alpha(B^*B)B^*BL^s x^\dag-x^\dag}_r &(\text{because }B^*y=B^*BL^s x^\dag)\\
&= \norm{L^{-s}[g_\alpha(B^*B)B^*B-I]L^sx^\dag}_r&\\
&= \norm{L^{-s}r_\alpha(B^*B)L^s x^\dag}_r&\\
&= \norm{L^{-s}r_\alpha(B^*B)(B^*B)^{\frac{u-s}{2(a+s)}}v}_r & \left(\text{by } (\ref{eq:8 teo 8.23})\right)\\
&=\norm{L^{-s}(B^*B)^{\frac{u-s}{2(a+s)}}r_\alpha(B^*B)v}_r & \\
&=\norm{(B^*B)^{\frac{u-s}{2(a+s)}}r_\alpha(B^*B)v}_{r-s}&\\
&\leq m^{\frac{r-s}{a+s}}\norm{(B^*B)^{\frac{s-r}{2(a+s)}}(B^*B)^{\frac{u-s}{2(a+s)}}r_\alpha(B^*B)v} &  \left(\text{by (\ref{eq:cota B*B}) with } \nu=\frac{s-r}{a+s}\right)\\
&= m^{\frac{r-s}{a+s}} \norm{(B^*B)^{\frac{u-r}{2(a+s)}}r_\alpha(B^*B)v}&\\
&\leq m^{\frac{r-s}{a+s}}\,c_{\bar{\mu}}\alpha^{\frac{u-r}{2(a+s)}}\norm{v}  & \hskip -.9in {\left(\text{by (\ref{eq:calif}) with } \bar{\mu}\doteq\frac{u-r}{2(a+s)},\,0\leq\bar{\mu}\leq 1\right)}\\
&= m^{\frac{r-s}{a+s}}\,c_{\bar{\mu}} \left[c\left(\frac{\delta}{\norm{x^\dag}_u}\right)^{\frac{2(a+s)}{a+u}}\right]^{\frac{u-r}{2(a+s)}}\norm{v}& \\
&=m^{\frac{r-s}{a+s}}\,c_{\bar{\mu}}\, c^{\frac{u-r}{2(a+s)}}\delta^{\frac{u-r}{a+u}}\norm{v}\norm{x^\dag}_u^{\frac{r-u}{a+u}}& \\
&=m^{\frac{r-s}{a+s}}\,c_{\bar{\mu}}\, c^{\frac{u-r}{2(a+s)}}\delta^{\frac{u-r}{a+u}}\norm{(B^*B)^{\frac{s-u}{2(a+s)}}L^sx^\dag}\norm{x^\dag}_u^{\frac{r-u}{a+u}}& \left(\text{by } (\ref{eq:8 teo 8.23})\right) \\
&\leq m^{\frac{r-s}{a+s}}\,c_{\bar{\mu}}\,(c+1)(M+1) \norm{x^\dag}_u^{\frac{a+r}{a+u}}\delta^{\frac{u-r}{a+u}}.&\left(\text{by (\ref{eq:cota B*B}) with } \nu=\frac{s-u}{a+s}\right)\\
\end{align*}
Thus there exists $C_2\doteq
m^{\frac{r-s}{a+s}}\,c_{\bar{\mu}}\,(c+1)(M+1)$ such that
\begin{equation}\label{eq:cota reg norma r}
\norm{x_\alpha-x^\dag}_r\leq C_2 \norm{x^\dag}_u^{\frac{a+r}{a+u}}\delta^{\frac{u-r}{a+u}}.
\end{equation}
\noindent Finally, from (\ref{eq:cota ruido norma r}) and
(\ref{eq:cota reg norma r}) it follows that there exists $C\doteq
C_1+C_2$ such that
$$\norm{x_\alpha^\delta-x^\dag}_r\leq C
\norm{x^\dag}_u^{\frac{a+r}{a+u}}\delta^{\frac{u-r}{a+u}},$$ as we
wanted to show.
\end{proof}

Regarding the estimate for the total error (\ref{eq:orden de conv
theorem normas r}) in the previous theorem it is important to note
the following: if $r>0$ then the order of convergence that we
obtain is worse than the one obtained in Theorem \ref{teo:ordenes
de convergencia} (see (\ref{eq:orden de conv theorem 8.23})), but
now this order is obtained in the stronger $\norm{\cdot}_r$ norm. On the
other hand if $r<0$, then $\norm{\cdot}_r$ is weaker than  $\norm{\cdot}$ and
therefore (\ref{eq:orden de conv theorem normas r}) provides an
estimate for the total error in a norm which is weaker than the
norm in $\X$. However, in this case it is important to note that
the order $\mathcal O \left(\delta ^{\frac{u-r}{a+u}}\right)$ in
(\ref{eq:orden de conv theorem normas r}) is now better than the
one obtained in (\ref{eq:orden de conv theorem 8.23}).

It is worth noting here that the parameter choice rule
(\ref{eq:elecalpha normas r}) requires of the explicit knowledge
of the degree of regularity $u$ of $x^\dag$. However, the
following result shows that convergence can also be obtained in
the norm $\norm{\cdot}_r$ when the parameter choice rule is chosen in
the form $\alpha=\delta^\varepsilon$, for $\varepsilon$ taking any
value within a certain interval.

\begin{thm}\label{teo:ext ordenes norma r}
Let $\X$, $\Y$, $T$, $T^\dag$, $L$, $(\X_t)_{t\in\mathbb R}$,
$s\geq 0$, $a>0$, $B=TL^{-s}$, $g_\alpha,$ $r_\alpha$,
$R_\alpha=L^{-s}g_\alpha(B^*B)B^*$, $y\in\D(T^\dag)$, $y^\delta
\in\Y$, $\|y-y^\delta\|\le\delta$, $u\in[0,a+2s],$ $x^\dag=T^\dag
y\in\X_u$, $x_\alpha=R_\alpha y$ y $x_\alpha^\delta=R_\alpha
y^\delta$, all as in Theorem \ref{teo:ordenes de convergencia}.
Let $r\in[-a, \min\{u,s\}]$ and suppose that the parameter choice
rule $\alpha$ is chosen as $\alpha=c\,\delta^\varepsilon$ where
$\varepsilon\in\left(0,\frac{2(a+s)}{a+r}\right]$. Then
\[\norm{x_\alpha^\delta-x^\dag}_r=\mathcal O(\delta^\sigma),\]
where $\sigma=\min\left\{
1-\frac{\varepsilon(a+r)}{2(a+s)},\frac{\varepsilon(u-r)}{2(a+s)}\right\}$.
The optimal order of convergence is obtained when $\varepsilon$ is
chosen to be $\varepsilon=\frac{2(a+s)}{a+u}$, in which case the
order of convergence (\ref{eq:orden de conv theorem normas r}) of
Theorem  \ref{teo:convergencia normas intermedias} is obtained.
\end{thm}

\begin{proof}
Following similar steps as in the proof on Theorem
\ref{teo:convergencia normas intermedias} it follows immediately
that
\[\norm{x_\alpha^\delta-x_\alpha}_r\leq
C_1\alpha^{-\frac{a+r}{2(a+s)}}\delta\hspace{0.5cm}\text{and}\hspace{0.5cm}
\norm{x_\alpha-x^\dag}_r\leq C_2 \alpha^{\frac{u-r}{2(a+s)}}.\]
Since $\alpha=c\,\delta^\varepsilon$ it then follows that
\begin{equation}\label{eq:1 en ext norma r}
\norm{x_\alpha^\delta-x_\alpha}_r\leq
C_1\,\delta^{1-\frac{\varepsilon(a+r)}{2(a+s)}}
\end{equation}
and
\begin{equation}\label{eq:2 en ext norma r}
\norm{x_\alpha-x^\dag}_r\leq C_2\,
\delta^{\frac{\varepsilon(u-r)}{2(a+s)}}.
\end{equation}
Form (\ref{eq:1 en ext norma r}) and (\ref{eq:2 en ext norma r})
it follows that
\[\norm{x_\alpha^\delta-x^\dag}_r=\mathcal O(\delta^\sigma),\]
where $\sigma=\min\left\{
1-\frac{\varepsilon(a+r)}{2(a+s)},\frac{\varepsilon(u-r)}{2(a+s)}\right\}$.
Also, from (\ref{eq:1 en ext norma r}) and (\ref{eq:2 en ext norma
r}) we also have that the order of convergence is optimal when
$\varepsilon$ is chosen such that
$$1-\frac{\varepsilon(a+r)}{2(a+s)}=\frac{\varepsilon(u-r)}{2(a+s)},$$
that is for $\varepsilon=\frac{2(a+s)}{a+u}$, in which case
$\sigma=\frac{u-r}{a+u}$.
\end{proof}

It is important to note now that the results of Theorems
\ref{coro:coro de ordenes de convergencia} and
\ref{teo:convergencia normas intermedias} are obtained for
particular choices  of the parameters in Theorem \ref{teo:ext
ordenes norma r}. In fact if $r=0$ then we obtain the convergence
result of Theorem \ref{coro:coro de ordenes de convergencia},
while for $\varepsilon=\frac{2(a+s)}{a+u}$ the convergence result
of Theorem \ref{teo:convergencia normas intermedias} is obtained.

In the next theorem we show that the optimal order of convergence
in Theorem \ref{coro:coro de ordenes de convergencia} can also be
achieved under the assumption of a source condition on $x^\dag$,
associated to the restriction of the operator $T$ to the Hilbert
scale $\X_s$, for some $s\geq 0$.

\begin{thm}\label{teo:ordenes de convergencia ext}
Let $\X$, $\Y$, $T$, $T^\dag$, $L$, $(\X_t)_{t\in\mathbb R}$,
$s\geq 0$, $a>0$, $\mu_0\ge 1$, $B=TL^{-s}$, $g_\alpha,$
$r_\alpha$, $R_\alpha=L^{-s}g_\alpha(B^*B)B^*$, $y\in\D(T^\dag)$,
$y^\delta \in\Y$, $\|y-y^\delta\|\le\delta$, $x^\dag=T^\dag y$,
$x_\alpha=R_\alpha y$ and $x_\alpha^\delta=R_\alpha y^\delta$, all
as in Theorem \ref{teo:ordenes de convergencia}. Suppose that
$x^\dag \in
\R\left(\left(L^{-2s}T^*T_{|_{\X_s}}\right)^{\frac{u-s}{2(a+s)}}\right)$
for some $u\in(s,2\mu_{\scriptscriptstyle 0}(a+s)-a]$ and that the
regularization parameter $\alpha$ is chosen as
\begin{equation}\label{eq:elecalpha theorem ext}
\alpha=c\left(
\frac{\delta}{\norm{x^\dag}_u}\right)^{\frac{2(a+s)}{a+u}}
\end{equation}
where $c>0$. Then there exists a contant $C$ (which depends on $a$
and $s$ but not on $u$) such that the following estimate for the
total error holds
\begin{equation}\label{eq:orden de conv theorem ext}
\norm{x_\alpha^\delta-x^\dag}\leq
C\delta^{\frac{u}{a+u}}.
\end{equation}
\end{thm}
\begin{proof}
Consider the operator
\begin{equation}\label{eq:op T restringido}
T_{|_{\X_s}}:(\X_s,\norm{\cdot}_s)\longrightarrow\Y.
\end{equation}
Observe that $\forall\,x\in\X_s$, $y\in\Y$ we have
\begin{align*}
\langle x,L^{-2s}T^*y\rangle_s&= \langle L^sx,L^{-s}T^*y\rangle&\\
&= \langle x,T^*y\rangle&\\
&=\langle Tx,y\rangle.&
\end{align*}
It then follows that the adjoint $T^\sharp$ of the operator
$T_{|_{\X_s}}$ defined in (\ref{eq:op T restringido}) is given by
$T_{|_{\X_s}}=L^{-2s}T^*$. Hence, the source condition $x^\dag\in
\R\left(\left(L^{-2s}T^*T_{|_{\X_s}}\right)^{\frac{u-s}{2(a+s)}}\right)$
can also be written as $x^\dag\in \R\left(\left(T^\sharp
T\right)^{\frac{u-s}{2(a+s)}}\right)$, that is $x^\dag=
\left(T^\sharp T\right)^{\frac{u-s}{2(a+s)}}v$ for some $v\in
\X_s$.

On the other hand
\begin{align}
R_\alpha &= L^{-s}g_\alpha(B^*B)B^* \nonumber\\
&= L^{-s}B^*g_\alpha(BB^*) \nonumber\\
&= L^{-s}L^{-s}T^*g_\alpha(TL^{-s}L^{-s}T^*)  \nonumber\\
&= T^\sharp g_\alpha(TT^\sharp )  \nonumber\\
&= g_\alpha(T^\sharp T)T^\sharp,\label{eq:Ralpha espectral}
\end{align}
and therefore the family of operators $R_\alpha$ constitutes a
spectral regularization for the operator $T_{|_{\X_s}}$ given in
(\ref{eq:op T restringido}).

Observe now that
\begin{align}
\norm{x_\alpha ^\delta-x^\dag}&\leq \norm{x_\alpha^\delta -x^\dag}^{\frac{s}{a+s}}_{-a}
\norm{x_\alpha^\delta -x^\dag}^{\frac{a}{a+s}}_{s} \nonumber\\
&\leq m^{-1} \norm{T(x_\alpha^\delta -x^\dag)}^{\frac{s}{a+s}}\norm{x_\alpha^\delta -x^\dag}^{\frac{a}{a+s}}_{s},
 \label{eq:cotas}
\end{align}
where the first inequality follows from (\ref{eq:des-interp-Esc
Hilb}) with $q=-a $ and $ r=0$ and the second one from (\ref{eq:LT
suavidad}).

For the first factor in the RHS of (\ref{eq:cotas}) we have the
estimate
\begin{align*}
\norm{T(x_\alpha^\delta -x^\dag)} &\leq \norm{T(x_\alpha^\delta -x_\alpha)}+\norm{T(x_\alpha -x^\dag)}\\
& \leq k\delta + \tilde c\alpha^{\hat u +1/2},
\end{align*}
with $\hat u\doteq \frac{u-s}{2(a+s)}$, $\tilde c=\norm{v}$ and
$k$ as in (\ref{eq:cota galpha}), where the last inequality
follows immediately from (\ref{eq:Ralpha espectral}) and from
Theorems 4.2 and 4.3 in \cite{refb:Engl-Hanke-Neubauer-1996} (note
that $0<\hat u\leq \mu_{\scriptscriptstyle 0}-\frac{1}{2})$. Then,
with $\alpha$ as in (\ref{eq:elecalpha theorem ext}) it follows
that
\begin{align}\label{eq:no se}
\norm{T(x_\alpha^\delta -x^\dag)}&\leq k \delta + \tilde c
\left(\norm{x^\dag}_u^{\frac{-2(a+s)}{a+u}}\right)^ {(\hat u
+1/2)} \left(\delta^{\frac{2(a+s)}{a+u}}\right)^{\hat u +1/2}\\
&=(k+ \tilde c \norm{x^\dag}_u^{-1})\,\delta\\
&\le \tilde C\,\delta,
\end{align}
where $\tilde C\doteq k+ \tilde c
(1+\gamma^{(1-2\mu_0)(a+s)})\norm{x^\dag}_s^{-1}$. Note here that
$\tilde C$ is independent of $u$.

On the other hand, for the second factor in (\ref{eq:cotas}), from
Corollary  4.4 in \cite{refb:Engl-Hanke-Neubauer-1996} with
$\mu=\frac{u-s}{2(a+s)}$, we get the estimate
\begin{equation}\label{eq:no se 1}
\norm{x_\alpha^\delta -x^\dag}_{s}\leq c\,\delta ^{\frac{2\mu}{2\mu+1}}= c\, \delta^{\frac{u-s}{a+u}},
\end{equation}
where $c>0$.

Finally, with the estimates (\ref{eq:no se}) and (\ref{eq:no se
1}) in (\ref{eq:cotas}) we obtain that
$$ \norm{x_\alpha ^\delta-x^\dag}\leq m^{-1}\left( \tilde C\delta\right)^{\frac{s}{a+s}}
\left(c\,\delta^{\frac{u-s}{a+u}}\right)^{\frac{a}{a+s}}= \hat C
\delta^{\frac{u}{a+u}},$$
where $\hat C\doteq m^{-1}\tilde
C^{\frac{s}{a+s}}c^{\frac{a}{a+s}}$. This concludes the proof.
\end{proof}

In the next theorem we will show that under the same conditions of
Theorem \ref{teo:ordenes de convergencia ext}, with the additional
hypotheses that the operators $L^{-1}$ and $T^*T$ commute, it is
possible to obtain the same order of convergence as in
(\ref{eq:orden de conv theorem ext}), but now for a larger range
of values of $u$.

\begin{thm}\label{teo: T y L conmutan}
Let $\X$, $\Y$, $T$, $T^\dag$, $L$, $s\geq 0$, $a>0$, $\mu_0\ge
1$, $B=TL^{-s}$, $g_\alpha,$ $r_\alpha$,
$R_\alpha=L^{-s}g_\alpha(B^*B)B^*$, $y\in\D(T^\dag)$, $y^\delta
\in\Y$, $x^\dag=T^\dag y$, $x_\alpha=R_\alpha y$ y
$x_\alpha^\delta=R_\alpha y^\delta$ y $\alpha=\alpha(\delta)$, all
as in Theorem \ref{teo:ordenes de convergencia}. Suppose also that
the operators $L^{-1}$ and $T^*T$ commute and that $x^\dag \in
\R\left(\left(B^*B\right)^{\frac{u}{2(a+s)}}\right)$ for some
$u\in[0,2\mu_{\scriptscriptstyle 0}(a+s)]$. Then there exists a
constant $C$ (which depends on $a$ and $s$ but not on $u$) such
that the following estimate for the total error holds:
\begin{equation*}
\norm{x_\alpha^\delta-x^\dag}\leq
C\delta^{\frac{u}{a+u}}.
\end{equation*}
\end{thm}
\begin{proof}
To prove this result we will follow similar steps as those in the
previous theorems, proceeding to estimate the error due to noise
and the regularization error separately. Just like in Theorem
\ref{coro:coro de ordenes de convergencia}, without loss of
generality we will assume that $y\in \R(T)$. For the error due to
noise, with the same proof as in Theorem \ref{coro:coro de ordenes
de convergencia}, from (\ref{eq:cota ruido}) we have that
\begin{equation}\label{eq:cota ruido theorem ext conm}
 \norm{x_\alpha^\delta
- x_\alpha}\leq C_1 \delta\alpha^{- \frac{a}{2(a+s)}},
\end{equation}
where $C_1=k\,m^{-\frac{s}{a+s}}$ with $k$ as in (\ref{eq:cota
galpha}) and  $m$ as in (\ref{eq:LT suavidad}).

On the other hand, since $L^{-1}$ commutes with $T^\ast T$, it
follows that $L^{-s}$ commutes with $B^\ast B$ and therefore, with
any function of $B^\ast B$. Let $v\in \X$ such that
$x^\dagger=(B^\ast B)^{\frac{u}{2(a+s)}}v$. Then for the
regularization error we have that
\begin{align}
\norm{x_\alpha-x^\dag}=&\norm{R_\alpha y -
x^\dag}& \nonumber \\
=&\norm{L^{-s}g_\alpha(B^*B)B^*y -x^\dag}&\nonumber\\
=& \norm{\left(L^{-s}g_\alpha(B^*B)B^*BL^s-I\right) x^\dag} &(\text{since }B^*y=B^*BL^sx^\dag)\nonumber\\
=&\norm{\left(g_\alpha(B^*B)B^*BL^{-s}L^s -I\right)x^\dag}&(\text{since }L^{-s} \text{ commutes with } T^*T)\nonumber\\
=& \norm{\left(g_\alpha(B^*B)B^*B-I\right)(B^*B)^{\frac{u}{2(a+s)}}v}
&\left(\text{since }x^\dag \in\R\left((B^*B)^\frac{u}{2(a+s)}\right)\right)\nonumber\\
=&\norm{r_\alpha(B^*B)(B^*B)^{\frac{u}{2(a+s)}}v} &\nonumber\\
=&\norm{(B^*B)^{\frac{u}{2(a+s)}}r_\alpha(B^*B)v} &\nonumber\\
\leq &c_{\bar\mu}\alpha ^{\frac{u}{2(a+s)}}\norm{v}&\left(\text{by
(\ref{eq:calif}) with }
{\bar\mu}\doteq\frac{u}{2(a+s)}\right)\nonumber\\
\doteq & C_2\alpha^{\frac{u}{2(a+s)}}.\label{eq:cuarentaydos} &
\end{align}
Thus
\begin{equation}\label{eq:cota reg theorem ext conm}
\norm{x_\alpha-x^\dag}\leq C_2\,\alpha ^{\frac{u}{2(a+s)}}.
\end{equation}
Finally from (\ref{eq:cota ruido theorem ext conm}) and
(\ref{eq:cuarentaydos}) it follows that
\begin{align*}
\norm{x_\alpha^\delta-x^\dag}&\leq C_1 \delta\alpha^{- \frac{a}{2(a+s)}}+C_2\,\alpha ^{\frac{u}{2(a+s)}}&\\
&= C_1 \delta^{\frac{u}{a+u}} \norm{x^\dag}^{\frac{a}{a+u}}+C_2\delta^{\frac{u}{a+u}}\norm{x^\dag}^{\frac{-u}{a+u}}&\\
&= \left(
C_1\norm{x^\dag}^{\frac{a}{a+u}}+C_2\norm{x^\dag}^{\frac{-u}{a+u}}\right)\delta^{\frac{u}{a+u}}\\
&\doteq C\,\delta^{\frac{u}{a+u}}.
\end{align*}
\end{proof}

In the table below and in Figure \ref{fig:valores de u} we
illustrate the restrictions on the parameter $u$ and the source
condition for $x^\dag$ guaranteeing the order of convergence given
in (\ref{eq:orden de conv theorem 8.23}). These results where
obtained in Theorems \ref{teo:ordenes de convergencia},
\ref{teo:ordenes de convergencia ext} and \ref{teo: T y L
conmutan} respectively. \vspace{0.5cm}
\begin{center}
   \begin{tabular}{| l | c |}
    \hline
     Source condition& Restriction on $u$ \\ \hline
 \textcolor[rgb]{0.00,0.00,1.00}{
$x^\dag\in\X_u$} & \textcolor[rgb]{0.00,0.00,1.00}{
$0\leq u\leq a+2s$} \\
\hline \textcolor[rgb]{0.12,0.38,0.14}{
$x^\dag\in\R\left(\left(L^{-2s}T^*T|_{\X_s}\right)^\frac{u-s}{2(a+s)}\right)$}
&
\textcolor[rgb]{0.12,0.38,0.14}{ $s< u\leq 2\mu_0(a+s)-a$} \\
\hline
     \textcolor[rgb]{1,0,0}{
$x^\dag\in\R\left((B^*B)^\frac{u}{2(a+s)}\right)$}
\textcolor[rgb]{1,0,0}{\text{and }$L^{-s}T^*T=T^*TL^{-s}$}
&\vspace{0.1cm}\textcolor[rgb]{1.00,0.00,0.00}{
$0\leq u\leq 2\mu_0(a+s)$}\vspace{-0.1cm}\\
\hline
   \end{tabular}
 \end{center}
\begin{figure}[H]
\psset{xunit=2cm, yunit=2cm}
\begin{pspicture}(2.5,-0.1)(-0.5,0.5)
\psline{|->}(0,0)(6,0) \rput[c](0,-0.2){{\small $0$}}
\psline{-|}(0,0)(1,0) \psline{-|}(0,0)(3,0) \psline{-|}(0,0)(5.5,0)
\psline[linewidth=2pt,linecolor=blue](0,0.1)(1,0.1)\textcolor[rgb]{0.00,0.00,1.00}{\rput[c](1.05,-0.2){{\small
$a+2s$}}}
\psline[linewidth=2pt,linecolor=green!50!black](0.2,0.2)(3,0.2)
\textcolor[rgb]{0.12,0.38,0.14}{\rput[c](3,-0.2){{\small
$2\mu_0(a+s)-a$}}}
\psline{-|}(0,0)(0.2,0)
\rput[c](0.2,-0.2){{\small
$s$}}

\psline[linewidth=2pt,linecolor=red](0,0.3)(5.5,0.3)\textcolor[rgb]{1.00,0.00,0.00}{\rput[c](5.5,-0.2){{\small
$2\mu_0(a+s)$}}}
\end{pspicture}
\caption{Possible values of the parameter ``$u$''.}\label{fig:valores de u}
\end{figure}
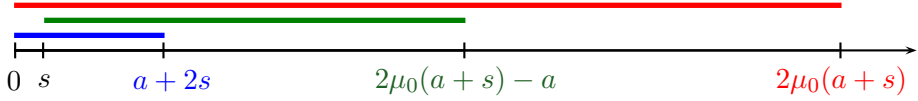
\vspace{1cm}
In the following proposition, a relation between the source sets
of Theorems \ref{teo:ordenes de convergencia} and \ref{teo:ordenes
de convergencia ext} is shown.
\begin{prop}\label{lem:relacion conj fuentes}
Let $\X,\,\Y,\,T,\,L$, $(\X_t)_{t\in\mathbb
R}$, $\,s\geq 0,\,a>0$ and $B= TL^{-s}$, all as in
Theorem \ref{teo:ordenes de convergencia ext}. Then, for every
$u\in[s, a+2s]$ there holds
\begin{equation}\label{eq:inclusion I1}
\X_u\subset\R\left(\left(L^{-2s}T^*T_{|_{\X_s}}\right)^{\frac{u-s}{2(a+s)}}\right),
\end{equation}
For $u=a+2s$ the inclusion in (\ref{eq:inclusion I1}) is in fact
an equality.
\end{prop}
\begin{proof}
Let $T^\sharp=L^{-2s}T^*$ the adjoint of the operator
$T_{|_{\X_s}}$ as defined in (\ref{eq:op T restringido}). Then,
for every $x\in\X_s$ we have that
\begin{equation*}
\norm{\left(T^\sharp T_{|_{\X_s}}\right)^{1/2}x}^2_s= \left\langle T^\sharp T_{|_{\X_s}} x,x \right\rangle_s
= \norm{T_{|_{\X_s}}x}^2.
\end{equation*}
From this equality and (\ref{eq:LT suavidad}) it follows that
\begin{equation}\label{eq:3-lema relacion fuentes}
 m\norm{x}_{-a}   \leq  \norm{ (T^\sharp T_{|_{\X_s}})^{1/2}x}_s \leq  M\norm{x}_{-a} \hspace{0.4cm}\forall \,\,x\in\X_s.
\end{equation}

On the other hand, note that
\begin{align}
\D((T^\sharp T_{|_{\X_s}})^{-1/2})&=\R((T^\sharp T_{|_{\X_s}})^{1/2})\nonumber\\
&=\R(T^\sharp)\nonumber\\
&=\R(L^{-2s}T^*)\nonumber \\
&=\X_{a+2s},\label{eq:igualdad dominios para H}
\end{align}
where the last equality follows immediately from Lemma
\ref{lem:R(T*)=Xa}.

Now, using (\ref{eq:3-lema relacion fuentes}), (\ref{eq:igualdad
dominios para H}) and a duality argument it follows easily that
\begin{equation}\label{eq:4-lema relacion fuentes}
\frac{1}{M}\norm{x}_{a+2s}\leq\norm{(T^\sharp T_{|_{\X_s}})^{-1/2}x}_s\leq
\frac{1}{m}\norm{x}_{a+2s} \hspace{0.3cm}\forall\,\,x\in \X_{a+2s}.
\end{equation}
From (\ref{eq:igualdad dominios para H}) and (\ref{eq:4-lema
relacion fuentes}), the use of Heinz inequality (Theorem
\ref{teo:desigualdad de Heinz}) for the operators  $L^{a+2s}$ and
$(T^\sharp T_{|_{\X_s}})^{-1/2}$ allows us to conclude that for
every $\nu\in[0,1]$ there holds:
\begin{equation}\label{eq:inclusion dominios nu Heinz}
\D\left(L^{\nu(a+2s)}\right)=\D\left((T^\sharp T_{|_{\X_s}})^{-\nu/2}\right)
\end{equation}
and
\begin{equation*}
{M^{-\nu}}\norm{L^{\nu(a+2s)}x}\leq\norm{(T^\sharp T_{|_{\X_s}})^{-\nu/2}x} \leq {m^{-\nu}}
\norm{L^\nu{(a+2s)}x}\hspace{0.5cm}\forall\,\,x\in\D\left(L^{\nu(a+2s)}\right).
\end{equation*}

Finally we have that
\begin{align*}
\X_u& =\D(L^{u})&\\
&=\D((L^{a+2s})^\nu)& \left(\text{with }\nu\doteq\frac{u}{a+2s}\right)\\
&=\D((T^\sharp T_{|_{\X_s}})^{\frac{-u}{2(a+2s)}})&\left(\text{by (\ref{eq:inclusion dominios nu Heinz}) with }\nu\doteq\frac{u}{a+2s}\right)\\
&=\R((T^\sharp T_{|_{\X_s}})^{\frac{u}{2(a+2s)}})&\\
&\subset \R((T^\sharp T_{|_{\X_s}})^{\frac{u-s}{2(a+s)}}) &\left(\text{since }0\leq\frac{u-s}{a+s}\leq\frac{u}{a+2s}\right)\\
&= \R((L^{-2s}T^*T_{|_{\X_s}})^{\frac{u-s}{2(a+s)}}),&
\end{align*}
which proves the first part of the lemma.

For the second part, note that if $u=a+2s$ then
\begin{align*}
\X_u&=\X_{a+2s}&\\
&=\D((T^\sharp T_{|_{\X_s}})^{-1/2})&(\text{by (\ref{eq:igualdad dominios para H})})\\
&=\R((T^\sharp T_{|_{\X_s}})^{1/2})&\\
&=\R((T^\sharp T_{|_{\X_s}})^{\frac{u-s}{2(a+s)}})&\\
&=\R((L^{-2s}T^*T_{|_{\X_s}})^{\frac{u-s}{2(a+s)}}).&
\end{align*}
This completes the proof of the lemma.
\end{proof}

It is worth noting that the inclusion in (\ref{eq:inclusion I1})
reveals that the source condition $x^\dag
\in\R\left((L^{-2s}T^*T_{|_{\X_s}})^\frac{u-s}{2(a+s)}\right)$ in
Theorem \ref{teo:ordenes de convergencia ext} is less restrictive
than the source condition $x^\dag\in\X_u$ of Theorem
\ref{teo:ordenes de convergencia} for values of $u \in [s,\,
a+2s]$. Therefore, the latter theorem can now be seen as a
corollary of Theorem \ref{teo:ordenes de convergencia ext}.
Moreover, note that since $\mu_{\scriptscriptstyle 0}\geq 1$,
Theorem \ref{teo:ordenes de convergencia ext} is valid for $u$ in
a set which is larger than the one for which  Theorem
\ref{teo:ordenes de convergencia} holds. In light of this
observation it is then reasonable to question the relevance of
Theorem \ref{teo:ordenes de convergencia}. The answer to this
questioning is immediately answered by observing that the source
condition $x^\dag\in\X_u$, although less restrictive than  the
condition $x^\dag
\in\R\left((L^{-2s}T^*T_{|_{\X_s}})^\frac{u-s}{2(a+s)}\right)$ is,
in general, easier to verify since it involves only the operator
$L$ while the second involves both $L$ and $T$. On the other hand,
if the operators $L^{-1}$ and $T^*T$ commute, then there exist
close connections between the source conditions in Theorems
\ref{teo:ordenes de convergencia}, \ref{teo:ordenes de
convergencia ext} and \ref{teo: T y L conmutan}. We shall
establish these connections in Corollary \ref{lem: T y L conmutan
Xmu}. An extension of the second part of Theorem \ref{theorem:cota
B*B}, namely identity (\ref{eq:igualdad de rango BestrellaB en
teorema}) for values of $\nu>1$, will be previously needed. We
shall obtain such extension in Theorem \ref{theorem:extension
theorem B*B}. A few previous results, which are presented in the
next three lemmas, will be needed.

\begin{lem}\label{lem:X2a en rango}
Let $\X$, $\Y$, $T$, $L$, $(\X_t)_{t\in\mathbb R}$ and $s\geq 0$,
all as in Proposition \ref{lem:relacion conj fuentes}. Suppose
also that there exist positive constants $m,M$ with $0<m\leq
M<\infty$ and $a\in\mathbb R^+$ such that (\ref{eq:LT suavidad})
holds, i.e.
\begin{equation}
m\norm{x}_{-a}\leq\norm{Tx}\leq M \norm{x}_{-a}\hspace{0.3cm}\forall\,\, x\in\X,
\end{equation} then:
\begin{enumerate}[i)]
\item $\R(T^*T)\subset \X_{2a}$.
\item $\mathcal R(L^{-2s}T^*T) \subset\X_{2(a+s)}$.
\end{enumerate}
If $L^{-1}$ and $T^\ast T$ commute, then equality holds in both
inclusions above.
\end{lem}
\begin{proof}
To prove \textit{i)}, observe that since $\mathcal M\subset
\X_t\,\,\forall\,t$ and $\overline{\mathcal M}^{\norm{\cdot}_t}=\X_t$,
it follows immediately that $\overline{\mathcal
X_{2a}}^{\norm{\cdot}_a}=\X_a$. Suppose now that $x\in\R(T^*T)$. Then
from Lemma \ref{lem:R(T*)=Xa}, $x\in\X_a$. Hence, there exists a
sequence $\{x_n\}\subset\X_{2a}$ such that
$\norm{x_n-x}_a\rightarrow 0$ and therefore also
$\norm{x_n-x}\rightarrow 0$. Then $\norm{TL^{2a}x_n}\leq
M\norm{L^{2a}x_n}_{-a}\leq M\norm{L^ax_n}<\infty$. Thus, the
sequence $\{TL^{2a}x_n\}$ is bounded in $\Y$ and therefore there
exist $y\in\Y$ and a subsequence of $\{x_n\}$ (also denoted by
$\{x_n\}$) such that $TL^{2a}x_n\overset{w}{\longrightarrow}y$.
Finally, since the operator $TL^{2a}$ is closed, we have that
$x\in\D(L^{2a})=\X_{2a}$ and, moreover, $TL^{2a}x=y$. Thus
$\R(T^*T) \subset \X_{2a}$, which proves \textit{i)}.

Suppose now that $L^{-1}$ and $T^\ast T$ commute and let
$x\in\X_{2a}$. Then $L^a x\in \X_a=\R((T^\ast T)^{1/2}))$ (see
Lemma \ref{lem:R(T*)=Xa}), and therefore there exists $x_1\in\X$
such that $L^a x=(T^\ast T)^{1/2} x_1$. Then $x=L^{-a}(T^\ast
T)^{1/2} x_1=(T^\ast T)^{1/2} L^{-a} x_1$, where the last equality
holds by virtue of the commutativity of $L^{-1}$ and $T^\ast T$.
Now, since $L^{-a} x_1\in \X_a=\R((T^\ast T)^{1/2})$, it follows
that there exists $w\in \X$ such that $L^{-a}x_1=(T^\ast T)^{1/2}
w$. Finally then $x=T^\ast T w\in \R(T^\ast T)$ and hence,
equality holds in \textit{i)}.

To prove \textit{ii)}, let $x\in\mathcal R(L^{-2s}T^*T)$. Then
there exists $x_\ce\in\X$ such that $L^{-2s}T^*Tx_\ce=x.$ But from
\textit{i)} it follows that $T^*Tx_\ce\in\X_{2a}$ and therefore
$L^{-2s}T^*Tx_\ce\in\X_{2(a+s)}$. On the other hand if $L^{-1}$
and $T^\ast T$ commute and $x\in\X_{2(a+s)}\subset \X_{2s}$, then
there exists $L^{2s}x$ and $L^{2s}x\in\X_{2a}$. Since in this case
equality in \textit{i)} holds, it then follows that
$L^{2s}x\in\mathcal R(T^*T)$. Hence there exists $x_\ce\in\X$ such
that $L^{2s}x=T^*Tx_\ce$, and therefore
$x=L^{-2s}T^*Tx_\ce\in\R(L^{-2s}T^*T)$. This concludes the proof
of the lemma.
\end{proof}
\begin{lem}\label{coro:dominios}
Let $\X$, $\Y$, $T$, $L$, $(\X_t)_{t\in\mathbb R}$, $s\geq 0$,
$a>0$, $m$, $M$, all as in Lemma \ref{lem:X2a en rango}, and
$B=TL^{-s}$ as in Theorem \ref{coro:coro de ordenes de
convergencia}. If $L^{-1}$ and $T^*T$ commute then:
\begin{enumerate}[i)]
\item $\mathcal R(B^*B)=\X_{2(a+s)}$.
\item $m^{2}\,\|x\|_{-2(a+s)}\leq\|B^*Bx\|\leq M^{2}\,\|x\|_{-2(a+s)}\hspace{0.3cm}\forall
\,\,x\in\X$.
\item $M^{-2}\,\|x\|_{2(a+s)}\leq\|(B^*B)^{-1}x\|\leq m^{-2}\,\|x\|_{2(a+s)}\hspace{0.3cm}\forall
\,\,x\in\mathcal X_{2(a+s)}$.
\end{enumerate}
\end{lem}
\begin{proof}
Note that \textit{i)} follows immediately from the previous lemma.
To prove \textit{ii)} observe that for every $x\in\X$ we have
\begin{eqnarray*}
B^*Bx&=&L^{-2s}T^*Tx\\
&=& T^*TL^{-2s}x\\
&=&T^*TL^{2a}L^{-2a}L^{-2s}x\\
&=&(T^*T)^{1/2}(T^*T)^{1/2}L^{2a}L^{-2(a+s)}x.
\end{eqnarray*}
Thus
\begin{eqnarray*}
\norm{B^*Bx}&=&\norm{(T^*T)^{1/2}(T^*T)^{1/2}L^{2a}L^{-2(a+s)}x}\\
&=&\norm{T(T^*T)^{1/2}L^{2a}L^{-2(a+s)}x}\\
&\leq&M\norm{(T^*T)^{1/2}L^{2a}L^{-2(a+s)}x}_{-a}\hspace{1.5cm}(\text{from (\ref{eq:LT suavidad})})\\
&=&M\norm{(T^*T)^{1/2}L^{a}L^{-2(a+s)}x}\\
&=&M\norm{TL^{a}L^{-2(a+s)}x}\\
&\leq&M^2\norm{L^{a}L^{-2(a+s)}x}_{-a}\hspace{3cm} (\text{from (\ref{eq:LT suavidad})})\\
&=& M^2\norm{x}_{-2(a+s)}.
\end{eqnarray*}
Similarly, by using the inequality $m\norm{x}_{-a}\leq\norm{Tx}$,
it follows that $m^2\norm{x}_{-2(a+s)}\leq \norm{B^*Bx}$. This
completes the proof of \textit{ii)}.

To prove \textit{iii)} let $x\in\X_{2(a+s)}$. Then
\begin{eqnarray*}
\norm{(B^*B)^{-1}x}&=& \sup _{\underset{\norm {\bar x}= 1}{\bar
x\in\X_{2(a+s)}}}|\langle(B^*B)^{-1}x,\bar x\rangle|\\
&=&\sup _{\underset{\norm {\bar x}= 1}{\bar
x\in\X_{2(a+s)}}}|\langle x, (B^*B)^{-1}\bar x\rangle|\\
&=& \sup _{\underset{\norm{B^*B z}\leq 1}{z\in\X}}|\langle x, z\rangle|\\
&\leq& \sup _{\underset{m^2\norm{z}_{-2(a+s)}\leq 1}{z\in\X}}|\langle x, z\rangle|\hspace{1.8cm}(\text{from \textit{ii)}})\\
&=& \sup _{\underset{\norm{w}\leq 1}{w\in\X_{2(a+s)}}}|\langle x, m^{-2}L^{2(a+s)}w\rangle|\\
&=& m^{-2}\sup _{\underset{\norm{w}\leq 1}{w\in\X_{2(a+s)}}}|\langle L^{2(a+s)} x, w\rangle|\\
&=& m^{-2}\norm{x}_{2(a+s)}.\\
\end{eqnarray*}
A similar reasoning allows us to prove that
$M^{-2}\norm{x}_{2(a+s)}\leq \norm{(B^*B)^{-1}x}$. This concludes
the proof of the lemma.
\end{proof}

In the following lemma it is proved that, under the hypothesis of
commutativity of the operators $A$ and $L$, Heinz inequality
(Theorem \ref{teo:desigualdad de Heinz}) is also valid for
$\nu>1$.
\begin{lem}\label{lem:extension heinz dominios}
Let $A$ and $L$ be two unbounded, self-adjoint, strictly positive
operators on a Hilbert space $\X$. Suppose also that
$\D(A)\subset\D(L)$, $A$ and $L$ commute on $\D(A)$ and
$\norm{Lx}\leq \norm{Ax}$ for every $x\in\D(A)$. Then for every
$k\geq 0$ it follows that $\D(A^k)\subset\D(L^k)$ and
$\norm{L^kx}\leq \norm{A^kx}$ for every $x\in\D(A^k)$.
\end{lem}
\begin{proof}
If $0\leq k\leq 1,$  the result is true by virtue of Heinz
inequality (Theorem \ref{teo:desigualdad de Heinz}). Suppose then
that $k>1$. We will prove first that the result is true for all
$k\in\mathbb N$, that is, we will first show, by induction, that
$\D(A^n)\subset \D(L^n)$ and $\norm{L^nx}\leq\norm{A^nx}\,\,
\forall\, x\in\D(A^n),\,\forall\, n\in\mathbb N$. For that let
$n=2$ and $x\in \D(A^2)$. Since
$x\in\D(A^2)\subset\D(A)\subset\D(L)$, there exists $w\doteq Lx$.
On the other hand, since $x\in\D(A^2)$, $Ax\in\D(A)\subset\D(L)$
and therefore there exists $z\doteq LAx$. Thus
$$z=LAx=ALx=Aw.$$
Then $w\in\D(A)\subset\D(L)$ and therefore there exists $r\in\X$
such that $r=Lw=L^2x$. Hence $x\in\D(L^2)$. We have proved that
$\D(A^2)\subset\D(L^2)$. Also, for $x\in\D(A^2)$ we have that
$\norm{L^2x}\leq\norm{ALx}=\norm{AL x}\leq \norm{A^2x}$.

Suppose now that $\D(A^n)\subset\D(L^n)$ and $\norm{L^nx}\leq
\norm{A^nx}\forall\,x\in \D(A^{n})$. We will prove the inclusion
for $n+1$. For that let $x\in\D(A^{n+1})$. Since
$\D(A^{n+1})\subset\D(A^n)\subset\D(L^n)$, there exists $w\doteq
L^n x$. On the other hand, $Ax\in\D(A^n)$ and by the inductive
hypothesis $\D(A^n)\subset \D(L^n)$. Then there exists $z\doteq
L^n Ax$. Thus
$$z=L^nAx= A L^n x = Aw.$$
Then $w\in\D(L)$ and therefore there exists $r\doteq
Lw=LL^nx=L^{n+1}x$. Hence $x\in\D(L^{n+1})$. Also, if
$x\in\D(A^{n+1})$ then $\norm{L^{n+1}x}=\norm{L^nLx}\leq
\norm{A^nLx}=\norm{LA^n x}\leq \norm{A^{n+1}x}$.

We have then proved that for every $n\in\mathbb N$
\begin{equation}\label{eq:inclusion dominios naturales}
\D(A^{n})\subset\D(L^n) \hspace{0.5cm}\text{and}\hspace{0.5cm}
\norm{L^nx}\leq \norm{A^nx} \hspace{0.5cm}\forall\,x\in\D(A^n).
\end{equation}

Suppose now that $k\in\mathbb R^+\setminus \mathbb N$ and define
$n\doteq \lceil k\rceil$ (where ``$\lceil\cdot\rceil$'' denotes
the ``ceiling" function). Since $n\in\mathbb N$, from
(\ref{eq:inclusion dominios naturales}) we have that
$\D(A^n)\subset\D(L^n)$ and $\norm{L^nx}\leq \norm{A^nx}$. Now, by
using Theorem \ref{teo:desigualdad de Heinz} with $L$ and $A$
replaced by $L^n$ and $A^n$ and $\nu\doteq \frac{k}{\lceil
k\rceil}$, it follows that $\D(A^\nu)\subset\D(L^\nu)$ and
$\norm{L^\nu x}\leq \norm{A^\nu x}\,\,\forall\,x\in\D(A^\nu)$,
that is
$$\D(A^k)\subset\D(L^k)\hspace{0.5cm}\text{and}\hspace{0.5cm}
\norm{L^kx}\leq \norm{A^kx} \hspace{0.5cm}\forall\,x\in\D(A^k).$$
\end{proof}
Having proved the three previous lemmas, we are now ready to prove
an extension of the identity (\ref{eq:igualdad de rango BestrellaB
en teorema}) of Theorem \ref{theorem:cota B*B}, which will allow
us to show the relationships between the source conditions of
Theorems \ref{teo:ordenes de convergencia} and \ref{teo: T y L
conmutan}, that is, conditions of the form $x^\dag\in\X_u$ and
$x^\dag \in\R\left((B^*B)^\frac{u}{2(a+s)}\right)$,  for the case
in which $L^{-1}$ y $T^*T$ commute.
\begin{thm}\label{theorem:extension theorem B*B}
Let $T:\X\longrightarrow\Y$ be a linear continuous operator
between the Hilbert spaces $\X$ and $\Y$, $L$ a linear, densely
defined, unbounded and strictly positive operator on $\X$, and
$(\X_t)_{t\in\mathbb R}$ the Hilbert scale induced by $L$ over
$\X$. Let also $s$ be a positive constant, $B=TL^{-s}$ and suppose
that there exist positive constants $a,\,m$ and $M$ such that
(\ref{eq:LT suavidad}) holds. Assume also that $L^{-1}$ and $T^*T$
commute. Then for every $\nu>0$ we have that
\begin{equation}\label{eq:igualdad de rango BestrellaB en teorema
ultimo}
\R\left((B^*B)^\frac{\nu}{2}\right)=\X_{\nu(a+s)}.
\end{equation}
\end{thm}
\begin{proof}
First note that from Lemma \ref{coro:dominios} \textit{i)} it
follows that $\mathcal D((B^*B)^{-1})=\X_{2(a+s)}=\mathcal
D(L^{2(a+s)})$. On the other hand since the operators $L^{-1}$ and
$T^*T$ commute, then $T^*T$ and $L^{-r}$ also commute for every
$r>0$ (see \cite{refb:Friedrichs-1973}, page 140). Then, the
operators $B^*B=L^{-s}T^*TL^{-s}$ and $L^{-2(a+s)}$ commute and
therefore their respective inverses also commute. From Lemma
\ref{coro:dominios} \textit{iii)} and Lemma \ref{lem:extension
heinz dominios} it then follows that
$$\mathcal D\left(\left((B^*B)^{-1}\right)^\nu\right)=\mathcal D\left(\left(L^{2(a+s)}\right)^\nu\right)
\hspace{0.3cm}\forall\,\, \nu\geq0,$$ that is
$$\mathcal R\left((B^*B)^\nu\right)=\mathcal \X_{2\nu(a+s)}.$$
\end{proof}
The following corollary shows the relation between the source
conditions of Theorems \ref{teo:ordenes de convergencia},
\ref{teo:ordenes de convergencia ext} and \ref{teo: T y L
conmutan}.
\begin{cor}\label{lem: T y L conmutan Xmu}
Let $\X$, $\Y$, $T$, $L$, $(\X_t)_{t\in\mathbb R}$, $s$, $a$ and
$B$ as in Theorem \ref{theorem:extension theorem B*B}. Then
\begin{enumerate}[i)]
\item $\X_u=\R\left((B^*B)^\frac{u}{2(a+s)}\right)\,\,\forall\,u\geq
0$.
\item $\R\left((L^{-2s}T^*T)^\frac{u}{2(a+s)}\right)\subset
\R\left((L^{-2s}T^*T_{|_{\X_s}})^\frac{u-s}{2(a+s)}\right)\,\,\forall\,u\in[s,a+2s]$.
\end{enumerate}
\end{cor}
\begin{proof}
Part \textit{i)} follows immediately from Theorem
\ref{theorem:extension theorem B*B} with $\nu\doteq\frac{u}{a+s}$.
To prove \textit{ii)} note that if $u\in[s,a+2s]$ then
\begin{align*}
\R\left((L^{-2s}T^*T)^\frac{u}{2(a+s)}\right)&= \R\left((B^*B)^\frac{u}{2(a+s)}\right) &
(\text{since } L^{-s} \text{ y }T^*T \text{ commute})\\
&=\X_u & (\text{by }\textit{i)})\\
&\subset
\R\left((L^{-2s}T^*T_{|_{\X_s}})^\frac{u-s}{2(a+s)}\right).&(\text{by
Proposition \ref{lem:relacion conj fuentes}})
\end{align*}
Hence
\begin{equation*}
\R\left((L^{-2s}T^*T)^\frac{u}{2(a+s)}\right)\subset
\R\left((L^{-2s}T^*T_{|_{\X_s}})^\frac{u-s}{2(a+s)}\right),
\end{equation*}
as we wanted to prove.
\end{proof}

\begin{rem}
Under the hypothesis that the operators $L^{-1}$ and $T^*T$
commute, Corollary \ref{lem: T y L conmutan Xmu} implies that for
$u\in[s, a+2s]$ the source condition $x^\dag
\in\R\left((B^*B)^\frac{u}{2(a+s)}\right) $ of Theorem \ref{teo: T
y L conmutan} is more restrictive than the source condition
$x^\dag \in\R\left((L^{-2s}T^*T_s)^\frac{u-s}{2(a+s)}\right) $ of
Theorem \ref{teo:ordenes de convergencia ext}. However, it is
important to point out here that Theorem \ref{teo: T y L conmutan}
is valid for a set of values of $u$ which is larger than the one
for which Theorem \ref{teo:ordenes de convergencia ext} is valid.
In particular, Theorem \ref{teo: T y L conmutan} is valid for
values of $u\in(2\mu_{\scriptscriptstyle
0}(a+s)-a,2\mu_{\scriptscriptstyle 0}(a+s)]$ (for which Theorem
\ref{teo:ordenes de convergencia ext} is not valid), thus allowing
us to obtain better orders of convergence.
\end{rem}
%


\section{Main Results}\label{sec:3}
\subsection{Multiple Hilbert scales}

In this section we will first introduce the concept of a multiple
(or vectorial) Hilbert scale. Then we shall define a
regularization method in these multiple Hilbert scales and prove
several convergence theorems, some of which generalize results
obtained in the previous section.

Let $T$ be a linear continuous operator between the Hilbert spaces
$\X$ and $\Y$. Consider also $N$ linear, densely defined,
unbounded, strictly positive, self-adjoint operators , with open
dense domains.
\begin{equation}\label{eq:op Li en cap4}
L_i: \mathcal D(L_i)\subset\X \longrightarrow \X, \hspace{0.2cm}
i=1,2,...,N.
\end{equation}
Thus, each $L_i$ is a closed operator on $\X$ satisfying:
$\D(L_i)=\D(L_i^*)$ is dense in $\X$,  $\langle L_ix,y\rangle =
\langle x, L_iy\rangle$ for every $x,y\in\D(L_i)$ and there exists
a positive constant $\gamma_i$ such that $\langle
L_ix,x\rangle\geq \gamma_i \norm{x}^2$ for every $x\in\D(L_i)$.

In what follows we shall obtain regularized solutions of the
ill-posed problem $Tx=y$, by means of the simultaneous use of the
$N$ Hilbert scales induced on $\X$ by the operators $L_i,$ $1\leq
i\leq N$. The motivation for this development comes mainly from
the idea of combining the advantages of the use of general
penalizing terms in Tikhonov-Phillips type methods (see
\cite{ref:Mazzieri-Spies-Temperini-2011}) with the adaptive
virtues that regularization in Hilbert scales possess in regard to
the order of convergence of the total error as a function of the
degree of regularity of the exact solution. In order to do that we
will previously need to extend some of the concepts and
definitions that were introduced in Section \ref{sec:1}.

For each index $i,$ $1\leq i\leq N,$ consider the set $\mathcal
M_i$ of all the elements $x\in\X$ for which all natural powers of
$L_i$ are defined, i.e. $\mathcal M_i\doteq \bigcap _{k=1} ^\infty
\D(L_i^k)$. As seen in Section \ref{sec:1}, $\mathcal M_i$ is
dense in $\X$, the powers $L_i^s$ are well defined on $\mathcal
M_i$ for every $s\in\mathbb R$ and $\mathcal M_i= \bigcap
_{s\in\mathbb R} \D(L_i^s)$.
Now, for each $i=1,2,\ldots,N$, and for each $s\in\mathbb R$ we
define the mapping  $\langle\cdot,\cdot\rangle_{L_i,s}:\mathcal
M_i\times \mathcal M_i\longrightarrow \mathbb{C}$ as $\langle x,
y\rangle_{L_i,s} \doteq \langle L_i^sx, L_i^sy\rangle$,
$x,y\in\mathcal M_i$. Again, as seen in Section \ref{sec:1},
$\langle\cdot,\cdot\rangle_{L_i,s}$ defines an inner product in
$\mathcal M_i$, which induces the norm $\norm x_{L_i,s} \doteq
\norm{L_i^sx}$, and $L_i$ induces on $\X$ a Hilbert scale, that we
shall denote with $(\X_{\s
t}^{\scriptscriptstyle{L_i}})_{t\in\mathbb R}$. Here $\X_{\s
t}^{\scriptscriptstyle{L_i}}$ is the completion of $\mathcal M_i$
in the $\norm\cdot_{L_i,t}$-norm.

Let us now consider the Hilbert space $\X^N$ consisting of $N$
copies of $\X$, i.e. $\ds \X^N\doteq\bigotimes_{i=1}^N\X$ with the
usual inner product in a product space. With the operators $L_{\s
1},L_{\s 2},\ldots,L_{\s N}$ given in (\ref{eq:op Li en cap4}) we
define the operator $\ds \vec L: \X^N\longrightarrow\X^N $ as
\begin{equation}\label{eq:L multiples}
\D(\vec L)\doteq\bigotimes_{i=1}^N\D(L_i),\quad \ds \vec L\doteq
\text{diag}(L_{\scriptscriptstyle 1}, L_{\scriptscriptstyle 2},
\ldots, L_{\scriptscriptstyle N}),
\end{equation}
so that for $\ds\vec x=\left( x_{\scriptscriptstyle
1},x_{\scriptscriptstyle 2}, \ldots, x_{\scriptscriptstyle N}
\right)^T \in \D(\vec L)$ one has
$\vec L\vec x\doteq \left( L_{\scriptscriptstyle
1}x_{\scriptscriptstyle 1}, L_{\scriptscriptstyle
2}x_{\scriptscriptstyle 2}, \ldots, L_{\scriptscriptstyle
N}x_{\scriptscriptstyle N} \right)^T.
$
Given the operator $\vec L$ defined as in (\ref{eq:L multiples})
and $\vec s\doteq (s_{\scriptscriptstyle 1},s_{\scriptscriptstyle
2},\ldots,s_{\scriptscriptstyle N})^T \in\mathbb R^{\s N}$, the
operator $\vec L^{\vec s}$ is defined as $\vec L^{\vec s}\doteq
\text{diag}(L_{\scriptscriptstyle 1}^{ s_1}x_{\scriptscriptstyle
1}, L_{\scriptscriptstyle 2}^{ s_2}x_{\scriptscriptstyle 2},
\ldots, L_{\scriptscriptstyle N}^{ s_N}x_{\scriptscriptstyle N})$,
i.e. $\D(\vec L^{\vec s})\doteq \bigotimes_{i=1}^N\D(L_i^{s_i})$
and for $\vec x=(x_{\s 1},x_{\s_2},\ldots,x_{\s N})^T\in\D(\vec
L^{\vec s})$
\begin{equation}\label{eq:potencia vectorial multiple}
\vec L^{\vec s}\vec x\doteq
\left( L_{\scriptscriptstyle 1}^{
s_1}x_{\scriptscriptstyle 1}, L_{\scriptscriptstyle 2}^{
s_2}x_{\scriptscriptstyle 2}, \ldots, L_{\scriptscriptstyle N}^{
s_N}x_{\scriptscriptstyle N} \right)^T.
\end{equation}
Now, for every $\vec s\doteq (s_{\s 1}, s_{\s 2},\ldots,s_{\s
N})^T\in\mathbb R^N$ and $x,y\in\widehat {\mathcal M}\doteq
\bigotimes_{i=1}^N\mathcal M_i$, we define $\ds \langle x,
y\rangle_{\s \vec L^{\vec s}} \doteq \sum_{i=1}^{N}\langle x_i,
y_i\rangle_{\s L_i^{s_i}}=\sum_{i=1}^N \langle L_i^{s_i}x_i,
L_i^{s_i}y_i\rangle$. It can be immediately seen that
$\langle\cdot,\cdot\rangle_{\s \vec L^{\vec s}}$ defines an inner
product in $\widehat {\mathcal M}$, which in turn induces the norm
$\|\vec x\|_{\s \vec L^{\vec s}}^2=\sum_{i=1}^N
\norm{L_i^{s_i}x_i}^2$.

\begin{defn}\label{def:escalas de Hilbert multiples}
(Multiple or vectorial Hilbert scale) Let $\vec L$ be as in
(\ref{eq:L multiples}), $\vec s\doteq (s_{\s 1}, s_{\s
2},\ldots,s_{\s N})^T\in\mathbb R^N$ and $\widehat {\mathcal M}=
\bigotimes_{i=1}^N\mathcal M_i$. The Hilbert space $\X_{\s\vec
s}^{\s\vec L}$ is defined as the completion of $\widehat {\mathcal
M}$ with respect to the norm $\norm{\cdot}_{\s \vec L^{\vec s}}$. The
family of spaces $(\X_{\s \vec s}^{\s\vec L})_{\s \vec s\in\mathbb
R^N}$ is called the vectorial Hilbert scale induced by $\vec L$
over $\X^N$. The operator $\vec L$ is called a ``generator" of the
Hilbert scale $(\X_{\s \vec s}^{\s\vec L})_{\s \vec s\in\mathbb
R^{N}}$.
\end{defn}

\begin{rem}
Since $\vec L$ is diagonal, it can be easily seen that $\Pi _j
\X_{\s\vec s}^{\s\vec L}=\X_{\s j}^{\s L_j}$ where $\Pi _j$ is the
$j$-th canonical projection of $\X^N$ onto $\X$. Moreover, for any
$t\in\mathbb R$ one has that $\X_{\s t}^{\s N}= \X_{\s
(t,t,\ldots,t)}^{\s\vec L}$, where $\X_{\s t}^{\s N}$ denotes the
usual (single) Hilbert scale induced by $\vec L$ on $\X^N$. Note
here that $\X_{\s t}^{\s N}$ is defined for any positive
self-adjoint operator $\vec L$ (not necessarily diagonal) on
$\X^N$. Also, although here we are only interested in the case in
which $\X_i=\X\,\,\forall i$, the definition of a multiple Hilbert
scale can be extended to the case of an arbitrary product Hilbert
space $\X=\bigotimes_{i=1}^N\X_i$ in an obvious way.
\end{rem}


\subsection{Regularization in multiple Hilbert scales with a single observation}
Next we proceed to define an operator $\vec B$ which will allow us
to define a regularization for $T^\dag$ in a multiple Hilbert
scale. Consider the Hilbert spaces $\ds
\X^N\doteq\bigotimes_{i=1}^N\X$ and $\ds
\Y^N\doteq\bigotimes_{i=1}^N\Y$, each one of them with the usual
inherited inner product. Given $\vec s \in\mathbb
R_{\scriptscriptstyle +}^N$ the operator $\ds \vec
B:\X^N\longrightarrow \Y^N$ is defined as $\vec B\doteq \vec T\vec
L^{-\vec s}$, where $\vec T:\X^N\to\Y^N$ is defined by  $\vec
T\doteq\text{diag}(T,T,\ldots,T)$. Thus, given $\ds \vec x\in
\X^N$
\begin{equation}\label{eq:B escalas multiples}
\vec B\vec x=\vec T\vec L^{-\vec s}\vec x= \left(T L_{\scriptscriptstyle 1}^{
-s_1}x_{\scriptscriptstyle 1}, T L_{\scriptscriptstyle 2}^{
-s_2}x_{\scriptscriptstyle 2}, \ldots,T L_{\scriptscriptstyle N}^{
-s_N}x_{\scriptscriptstyle N} \right)^T.
\end{equation}
From the properties of the operators $L_i$, it
follows immediately that the adjoint of $\vec B$ is given by $\vec B^*\doteq
\vec L^{-\vec s}\vec T^*$, where $\vec T^*\doteq \text{diag}(T^*,T^*,\ldots,T^*)$.
Thus, for every  $\ds\vec y\in \Y^N$ we have that
\begin{equation}\label{eq:B* multiples}
\vec B^*\vec y=\vec L^{-\vec s}\vec T^*\vec y= \left( L_{\scriptscriptstyle 1}^{
-s_1}T^*y_{\scriptscriptstyle 1}, L_{\scriptscriptstyle 2}^{
-s_2}T^*y_{\scriptscriptstyle 2}, \ldots,L_{\scriptscriptstyle N}^{
-s_N}T^*y_{\scriptscriptstyle N} \right)^T,
\end{equation}
and therefore for every $\vec x \in\X^N$ there holds
\begin{align}
\vec B^*\vec B\vec x&= \vec L^{-\vec s}\vec T^*\vec T \vec L^{-\vec s}\vec x\nonumber\\
&=\left(L_{\scriptscriptstyle 1}^{
-s_1}T^*T L_{\scriptscriptstyle 1}^{
-s_1}x_{\scriptscriptstyle 1}, L_{\scriptscriptstyle 2}^{
-s_2}T^*T L_{\scriptscriptstyle 2}^{
-s_2}x_{\scriptscriptstyle 2}, \ldots,L_{\scriptscriptstyle N}^{
-s_N}T^*T L_{\scriptscriptstyle N}^{
-s_N}x_{\scriptscriptstyle N} \right)^T\nonumber\\
&= \left(B_{\s 1}^*B_{\s 1}x_{\s 1},B_{\s 2}^*B_{\s 2}x_{\s 2},
\ldots,B_{\s N}^*B_{\s N}x_{\scriptscriptstyle N}\right)^T,\label{eq:B*B en multiples}
\end{align}
where $B_i\doteq TL_i^{-s_i}$ y $B_i^*=L_i^{-s_i}T^*$ is the
adjoint of the operator $B_i$ (compare with the definition of $B$
given in Theorem \ref{theorem:cota B*B}). Note that the operators
$\vec B^*\vec B$ and $B_i^*B_i$, for each index $i$, $1\leq i\leq
N$, are linear self-adjoint operators on the Hilbert spaces $\X^N$
and $\X$, respectively. As such, for each one of them there exists
a unique spectral family which allows them to be represented in
terms of the integral of the identity with respect to the
``operator valued measure" induced by that spectral family. We
shall denote with $\ds \{E^{\scriptscriptstyle \vec B^*\vec
B}_{\lambda}\}_{\lambda \in \mathbb{R}} $ and
$\{E^{\scriptscriptstyle B_i^*B_i}_{\lambda}\}_{\lambda \in
\mathbb{R}}$ the spectral families of the operators $\vec B^*\vec
B$ y $B_i^*B_i$, respectively (note that these families are
partitions of the identity on the spaces $\X^N$ and $\X$
respectively).

Let $g:\mathbb{R}\to\mathbb{R}$ be a piecewise continuous function
and consider the operators $g(\vec B^*\vec B):\X^N\longrightarrow\X^N$ and
$g(B^*_iB_i):\X\longrightarrow\X$, $1\leq i\leq N$. From
(\ref{eq:B*B en multiples}) it can be easily proved that
\begin{equation}\label{eq:rela galpha multiples}
\left(g(\vec B^*\vec B)\vec x\right)_i= g(B^*_iB_i)x_i,
\end{equation}
where $\vec x=\left(x_{\s 1},x_{\s 2},\ldots,x_{\s N}\right)^T$.

The next theorem states a convergence result which generalizes
Theorem \ref{coro:coro de ordenes de convergencia} to the case of
multiple Hilbert scales.
\begin{thm}\label{theorem:ordenes de convergencia multiples}
Let $T\in\mathcal L(\X,\Y)$ with $\X$ and $\Y$ Hilbert spaces,
$L_i: \mathcal D(L_i)\subset\X \longrightarrow \X$, $1\leq i \leq
N$, linear densely defined, self-adjoint and strictly positive
operators on $\X$, each one of them with open domain,
$L_i\geq\gamma_i$ for a constant $\gamma_i>0$,
and let $\vec L:\X^N\longrightarrow\X^N$
be as in (\ref{eq:L multiples}). Suppose also that for each index
$i$, $1\leq i\leq N$, there exist constants $m_i,\; M_i$, with
$0<m_i\leq M_i<\infty$, and $a_i>0$, such that for every $x\in\X$
the following condition holds:
\begin{equation}\label{eq:LiT suavidad}
m_i\norm{x}_{\s L_i,-a_i}\leq\norm{Tx}\leq M_i \norm{x}_{\s L_i,-a_i}.
\end{equation}
Let $\vec s=(s_{\scriptscriptstyle 1},s_{\scriptscriptstyle
2},\ldots,s_{\scriptscriptstyle N})^T \in\mathbb
R_{\scriptscriptstyle +}^N,\,\vec T=\text{diag}(T,T,\ldots,
T),\,\vec B\doteq \vec T \vec L^{-\vec s},\,\vec
\eta=(\eta_{\scriptscriptstyle 1},\eta_{\scriptscriptstyle
2},\ldots,\eta_{\scriptscriptstyle N})^T \in\mathbb
R_{\scriptscriptstyle +}^N$ such that $\ds\sum_{i=1}^N \eta_i=1$.
Also let $g_\alpha:[0,\|\vec B\|^2]\rightarrow \mathbb R,\,
\alpha>0$, be a family piecewise continuous real-valued functions
verifying the following conditions:
\begin{itemize}
\item[C1:]For every $\lambda \in(0,\|\vec B\|^2]$ there
holds $\lim_{\alpha\rightarrow 0^{\scriptscriptstyle
+}}g_\alpha(\lambda)=\frac{1}{\lambda}$. 
\item[C2:] There exists a constant $\hat c>0$ such that $\forall \,\, \lambda \in(0,\norm{\vec B}^2]$ and
$\forall\,\,\alpha>0$ there holds
$\abs{g_\alpha(\lambda)}\leq \hat c\alpha^{-1}$. 
\item[C3:] There exists $\mu_{\scriptscriptstyle 0}\geq 1$ such
that if $\mu\in[0,\mu_{\scriptscriptstyle 0}]$ then
$\lambda^\mu\abs{r_\alpha(\lambda)} \leq c_\mu\alpha^\mu\;\;
\forall \,\,\lambda \in(0,\|\vec B\|^2]$, where $c_\mu$ is a
positive constant and $r_\alpha(\lambda)\doteq 1-\lambda
g_\alpha(\lambda)$.
\end{itemize}
For $y\in\D(T^\dagger)$ and $y^\delta\in\Y$ with
$\norm{y-y^\delta}\leq \delta$, we define the regularized solution
of the problem $Tx=y$ with data $y^\delta$, as
\begin{equation}\label{eq:reg esc Hilbert multiples}
x_{\alpha} ^\delta\doteq \vec \eta{\,\s\bullet}\left(\vec L^{-\vec s} g_{\alpha}(\vec B^*\vec B)\vec B^*
\underline{\vec y}\,^\delta\right),
\end{equation}
$\ds\underline{\vec y}\,^\delta\doteq
(y^\delta,y^\delta,\ldots,y^\delta)^T\in\Y^N$. Suppose that for
each index $i$, $1\leq i\leq N$ there exists $u_i\in [0,a_i+2s_i]$
such that $x^\dag =T^\dag y\in \X_{\s u_i}^{\s L_i}$, i.e. $\vec
x^\dag\in\X_{\s \vec u}^{\s \vec L}$, where $\vec x^\dag
\doteq(x^\dag,x^\dag,\ldots,x^\dag)^T$, $\vec u\doteq
(u_{\scriptscriptstyle 1},u_{\scriptscriptstyle
2},\ldots,u_{\scriptscriptstyle N})^T$ and $(\X_{\s t}^{\s
L_i})_{\s t\in\mathbb R}$, $(\X_{\s \vec u}^{\s \vec L})_{\s \vec
u\in\mathbb R^N}$ are the Hilbert scale induced by $L_i$ over $\X$
and the multiple Hilbert scale induced by $\vec L$ over
$\X^N=\bigotimes_{i=1}^N\X$, respectively. Suppose that the
regularization parameter $\alpha$ is chosen as
\begin{equation}\label{eq:elecalphamultiples}
\alpha=\alpha(\delta)\doteq c\,\delta^\varepsilon
\hspace{0.5cm}\text{with }\varepsilon\in\left(0,\min_{\s 1\leq
i\leq N}\left\{\frac{2(a_i+s_i)}{a_i}\right\}\right),
\end{equation}
where $c>0$ and, for each index $i$, with $1\leq i\leq N$, $a_i$
is the constant in (\ref{eq:LiT suavidad}). Then:
\begin{enumerate}
\item[i)] $\displaystyle\lim_{\delta\to 0}
x_{\alpha(\delta)}^\delta\,=\, x^\dag$ and, moreover,
\item[ii)] the total error satisfies the following order of convergence: $\norm{x_\alpha^\delta -x^\dag}=\mathcal
O(\delta^\sigma)$ where $\sigma\doteq\ds\min_{\s i\leq i\leq
N}\min\left\{
1-\frac{a_i\varepsilon}{2(a_i+s_i)},\frac{u_i\varepsilon}{2(a_i+s_i)}\right\}$.
\item[iii)] The order of convergence of the total error in \textit{ii)} is optimal when in
(\ref{eq:elecalphamultiples}) the value of $\varepsilon$ is chosen
as
$$\ds\varepsilon=\left(\max_{\s 1\leq i\leq
N}\frac{a_i}{2(a_i+s_i)}+\min_{\s 1\leq i\leq
N}\frac{u_i}{2(a_i+s_i)}\right)^{-1},$$ in which case
$\ds\norm{x_\alpha^\delta -x^\dag}=\mathcal
O(\delta^{\sigma_\ce}),$ where
$\ds\sigma_\ce\doteq\s\frac{\ds\min_{\s 1\leq i \leq
N}\frac{u_i}{2(a_i+s_i)}}{\ds\min_{\s 1\leq i\leq
N}\frac{u_i}{2(a_i+s_i)}+ \ds\max_{\s 1\leq i\leq
N}\frac{a_i}{2(a_i+s_i)}}$.
\end{enumerate}
\end{thm}
\begin{proof}
Applying Theorem \ref{coro:coro de ordenes de convergencia} to
each operator $L_i$, $1\leq i\leq N$, since $\varepsilon\leq
\frac{2(a_i+s_i)}{a_i}$, with the choice of $\alpha$ as in
(\ref{eq:elecalphamultiples}) it follows that
\begin{equation}\label{eq:orden i de conv multiples}
\norm{x_{i,\alpha}^\delta-x^\dag}=\mathcal
O\left(\delta^{\sigma_i}\right),
\end{equation}
where
\begin{equation}\label{eq:reg i esc Hilbert}
x_{i,\alpha} ^\delta\doteq
L_i^{-s_i}g_\alpha(B_i^*B_i)B_i^*y^\delta \hspace{0.5cm} \text{and
}\hspace{0.5cm} \sigma_i=
\min\left\{1-\frac{a_i\varepsilon}{2(a_i+s_i)},\frac{u_i\varepsilon}{2(a_i+s_i)}\right\}.
\end{equation}
Then,
\begin{align*}
\norm{x_{\alpha}^\delta -x^\dag}&=\norm{\vec
\eta{\,\s\bullet}\left(\vec L^{-\vec s} g_{\alpha}(\vec B^*\vec B)\vec B^* \underline{\vec y}\,^\delta\right) -x^\dag}& \\
&=\norm{\sum_{i=1}^N \eta_i L_i^{-s_i}g_{\alpha}(B_i^*B_i)B_i^* y\,^\delta -x^\dag}& (\text{by (\ref{eq:rela galpha multiples})})\\
&=\norm{\sum_{i=1}^N \eta_i x_{i,\alpha} ^\delta-x^\dag}&(\text{by (\ref{eq:reg i esc Hilbert})})\\
&=\norm{\sum_{i=1}^N \eta_i \left(x_{i,\alpha}^\delta-x^\dag\right)}&\left(\text{since }\sum_{i=1}^N\eta_i=1\right)&\\
&\leq \sum_{i=1}^N \eta_i \norm{x_{i,\alpha}^\delta-x^\dag}&\\
&\leq \sum_{i=1}^N \eta_i \,c_i\,\delta^{\sigma_i} & (\text{by (\ref{eq:orden i de conv multiples})})\\
&\leq C\,\delta^{\scriptscriptstyle{\s\underset{1\leq i\leq N}{\min}\sigma_i}}&\\
&= C\, \delta^\sigma,&
\end{align*}
where $C$ is a positive constant (for instance for $\delta \in
[0,1]$, $C$ can be taken as $C=\underset{1\leq i\leq
N}{\max}c_i$). That proves \textit{i)} and \textit{ii)}. To prove
\textit{iii)} note that from Theorem \ref{coro:coro de ordenes de
convergencia}, more precisely from (\ref{eq:1 coro ordenes}),
there exist positive constants $c_i,\, d_i$, $1\leq i\leq N$, such
that
\begin{align}
\norm{x_{\alpha}^\delta -x^\dag}&\leq \sum_{i=1}^N \eta_i \norm{x_{i,\alpha} ^\delta-x^\dag}&\nonumber\\
&\leq \sum_{i=1}^N \eta_i \left(c_i\,\delta\alpha^{\frac{-a_i}{2(a_i+s_i)}}+d_i\alpha^{\frac{u_i}{2(a_i+s_i)}}\right)& (\text{by (\ref{eq:1 coro ordenes})})\nonumber\\
&\leq C_1\,\delta\alpha^{\s-\underset{1\leq i\leq N}{\max}\frac{a_i}{2(a_i+s_i)}}+ C_2\,
\alpha^{\s\underset{1\leq i\leq N}{\min}\frac{u_i}{2(a_i+s_i)}}&\nonumber\\
&= \hat{C_1}\, \delta^{\s 1-\underset{1\leq i\leq
N}{\max}\frac{\varepsilon a_i}{2(a_i+s_i)}}+\hat{C_2}\,
\delta^{\s\underset{1\leq i\leq N}{\min}\frac{\varepsilon
u_i}{2(a_i+s_i)}}, & (\text{by (\ref{eq:elecalphamultiples})})\label{eq:cota error total cuenta
mult}
\end{align}
where $C_i$ and $\hat{C_i}$ are generic positive constants.

Finally, from (\ref{eq:cota error total cuenta mult}) it follows
that the order of convergence of the total error is optimal when
$\varepsilon$ satisfies $1-\underset{\s 1\leq i\leq
N}{\max}\frac{\varepsilon a_i}{2(a_i+s_i)}= \underset{\s 1\leq
i\leq N}{\min}\frac{\varepsilon u_i}{2(a_i+s_i)}$, that is when
$\varepsilon$ is chosen as
$
\ds\varepsilon=\left(\max_{\s 1\leq i\leq
N}\frac{a_i}{2(a_i+s_i)}+\min_{\s 1\leq i\leq
N}\frac{u_i}{2(a_i+s_i)}\right)^{-1}
$
in which case, also from (\ref{eq:cota error total cuenta mult}),
it follows that $\ds\norm{x_\alpha^\delta -x^\dag}=\mathcal
O(\delta^{\sigma_\ce})$, where $\ds\sigma_\ce$ is given by
$\ds\sigma_\ce=\frac{\ds\min_{\s 1\leq i \leq
N}\frac{u_i}{2(a_i+s_i)}}{\ds\min_{\s 1\leq i\leq
N}\frac{u_i}{2(a_i+s_i)}+ \ds\max_{\s 1\leq i\leq
N}\frac{a_i}{2(a_i+s_i)}}$.
\end{proof}
\begin{rem}\label{obs:dato vectorial}
From (\ref{eq:rela galpha multiples}) it follows that the
regularized solution $\ds x_{\alpha} ^\delta= \vec
\eta{\,\s\bullet}\left(\vec L^{-\vec s} g_{\alpha}(\vec B^*\vec
B)\vec B^*\underline{\vec y}\,^\delta\right)$ defined in
(\ref{eq:reg esc Hilbert multiples}) can also be written in the
form $\ds x_{\alpha} ^\delta=\sum_{i=1}^N \eta_i
x_{i,\alpha}^\delta$ where
$x_{i,\alpha}^\delta=L_i^{-s_i}g_\alpha(B_i^*B_i)B_i^*y^\delta$ is
a single regularized solution of the problem $Tx=y$ in the Hilbert
scale of order $s_i$ induced by the operator $L_i$ on $\X$, so
that $x_{i,\alpha}^\delta\in \X^{L_i}_{s_i}$. Therefore
$x_{\alpha}^\delta$ is a convex combination of such solutions. In
contrast with what happens in the case $N=1$, where it is known
that the regularized solution is in $\D(L^s)$, here, the degree of
regularity of $\ds x_{\alpha} ^\delta$ is not explicitly known
since the Hilbert scales $\X^{L_i}_{s_i}$ are not necessarily
related.
\end{rem}

%
%

\subsection{Regularization in multiple Hilbert scales with multiple observations}
In Theorem \ref{theorem:ordenes de convergencia multiples} we
noted that, given a single noisy observation $y^\delta$, we
generated the ``observation vector'' $\underline{\vec
y}\,^\delta\in\Y^N$ by using $N$ copies of $y^\delta$. In practice
it may happen that $N$ different observations of $y$, say $y_{\s
1}^\delta,y_{\s 2}^\delta,\ldots, y_{\s N}^\delta$, such that
$\norm{y_{\s i}^\delta -y}\leq \delta\,\,\forall\,\,i=1,2,...,N$,
be available. In such a case we can use them to construct the
observation vector in the form $\vec y\,^\delta\doteq(y_{\s
1}^\delta,y_{\s 2}^\delta,\ldots, y_{\s N}^\delta)^T\in\Y^N$.
Defining now
\begin{equation*}
x_{\alpha} ^\delta\doteq \vec \eta{\,\s\bullet}\left(\vec L^{-\vec s}
g_{\alpha}(\vec B^*\vec B)\vec B^*\vec y\,^\delta\right),
\end{equation*}
(with $\vec \eta,\,\vec s,\,g_\alpha,\,\vec B,\,
\alpha=c\,\delta^\varepsilon$ as in Theorem \ref{theorem:ordenes
de convergencia multiples}) it can be easily seen that the same
results of Theorem \ref{theorem:ordenes de convergencia multiples}
remain true. In particular, we have that $\ds
\lim_{\delta\rightarrow 0^{\s +}}\norm{x_\alpha^\delta-x^\dag}=0$
and $\norm{x_\alpha^\delta-x^\dag}=\mathcal O(\delta^\sigma)$ with
$\ds \sigma= \min_{\s 1\leq i\leq N}\min\left\{
1-\frac{a_i\varepsilon}{2(a_i+s_i)},\frac{u_i\varepsilon}{2(a_i+s_i)}\right\}$.
However, in this case of regularization in multiple Hilbert scales
with multiple observations, it is also possible to utilize
different types of regularization methods (i.e. different
$g_\alpha$'s) for each one of the observations $y_i^\delta,\,
1\leq i\leq N$, in each one of the $N$ Hilbert scales, maintaining
the convergence to the exact solution and even improving the order
of convergence. This may be of particular interest when certain
``\textit{a-priori}'' knowledge about the $i^\text{th}$
observation suggests the use of certain type of regularization
method. In order to proceed with the formalization and
presentation of this result, we will previously need to extend the
definition of a ``function of a self-adjoint operator'' $f(A)$, to
the case in which $\vec f:\mathbb{R}\to\mathbb{R}^N$ is a
vector-valued function and $A$ is a self-adjoint operator in a
product space $\ds\X=\bigotimes_{i=1}^N \X_i$, where $\X_i$ is a
Hilbert space for every $i=1,2,\ldots,N$. Let $\vec f:\mathbb
R\longrightarrow \RE^N,\, \vec f=(f_{\s 1},f_{\s 2},\ldots,f_{\s
N})^T,$ $\vec f$ be piecewise continuous,
$\{E_\lambda^A\}_{\lambda\in\RE}$ the spectral family of $A$,
$E_\lambda^A:\X\longrightarrow\X$, $E_\lambda^A=\left(E_{\s
\lambda, 1}^A,E_{\s \lambda, 2}^A,\ldots,E_{\s \lambda,
N}^A\right)^T$, $E_{\s \lambda, i}^A:\X\longrightarrow\X_i$ (note
that $E_{\s \lambda, i}^A$ is the $i^\text{th}$ component of the
projection operator $E_{\lambda}^A$ on $\X$). We define the
operator $\vec f(A)$ as the spectral vector-valued integral
\begin{align}\label{eq:int-espec-vect}
\vec f(A)\,\vec x&=\int_{-\infty}^\infty
\vec f(\lambda)\odot dE_\lambda^A\vec x=\left(
\begin{array}{c}
\vdots \\
\ds\int_{-\infty}^\infty f_i(\lambda)\, dE_{\s\lambda,i}^A\,\vec x \\
\vdots \\
\end{array}
\right),
\end{align}
where ``$\odot$" denotes the Hadamard product, with domain given
by
$$\ds \D(\vec f(A))=\left\{\vec
x\in\X:\sum_{i=1}^N\int_{-\infty}^{\infty}f_i^2(\lambda)\,d\norm{E_{\s\lambda,i}^A\vec
x}^2<\infty \right\}.$$
It is important to note in (\ref{eq:int-espec-vect}) that in the
integral $\ds\int_{-\infty}^\infty f_i(\lambda)\,
d\,E_{\s\lambda,i}^A\,\vec x$, the family
$\{E_{\s\lambda,i}^A\}_{\lambda \in\RE}$ is not a spectral family
(in fact it is not a partition of unity but rather a parametric
family of canonical projections of a spectral family on the
product space $\X=\bigotimes_{i=1}^N\X_i$). However, under the
hypothesis of piecewise continuity of $\vec f$, it can be easily
seen that its existence is guaranteed by the classical theory
functional calculus. In fact, given any $i$, $1\leq i\leq N$, by
defining $\vec g:\RE\longrightarrow \RE^N$ as $\vec
g(\lambda)=(f_{\s i}(\lambda),f_{\s i}(\lambda),\ldots,f_{\s
i}(\lambda))^T$, since $\vec f$ is piecewise continuous,  so is
$\vec g$ and therefore the operator $\vec g(A)$ is well defined
and it is clear that for every $\vec x\in\D(\vec g(A))$ one has
that $\ds[\vec g(A)\vec x]_i=\int_{-\infty}^\infty f_i(\lambda)\,
dE_{\s\lambda,i}^A\,\vec x$.

With this extension of the concept of a function of an operator to
the case of vector-valued functions of self-adjoint operators on
product spaces, we are now ready to present the following theorem
which extends the result of Theorem \ref{theorem:ordenes de
convergencia multiples} to the case of multiple observations with
vector-valued regularization functions in multiple Hilbert scales.

\begin{thm}\label{theorem:ordenes de convergencia multiples vect}
Let $\X,\,\Y,\,\X^N,\,\Y^N,\,T,\,\vec T,\, L_i,\, (\X_{\s t}^{\s
L_i})_{\s t\in\RE},\,\vec s, \, B_i=TL_{\s i}^{\s -s_i},\, 1\leq
i\leq N$, $\vec L$,  $(\X_{\s \vec u}^{\s \vec L})_{\s\vec
u\in\RE^N},\,\vec B=\vec T\vec L^{-\vec s}$ and $\vec\eta$, all as
in Theorem \ref{theorem:ordenes de convergencia multiples}. For
each index $i$,  $1\leq i\leq N$, let
$g_{\alpha_i}^i:[0,\norm{B_i}^2]\rightarrow \mathbb R,\,
\alpha_i>0$ be a family of piecewise continuous functions and
$r_{\alpha_i}^i(\lambda)\doteq 1-\lambda g_{\alpha_i}^i(\lambda)$.
Suppose also that each one of the families $\{g_{\alpha_i}^i\}$
verifies the conditions \textit{C1}, \textit{C2} y \textit{C3} of
Theorem \ref{theorem:ordenes de convergencia multiples}, that is:
\begin{eqnarray*}
\text{C1}&:&\forall \,\,\lambda \in(0,\norm{B_i}^2] \text{ there holds }\lim_{\alpha_i\rightarrow 0^{\scriptscriptstyle +}}g_{\alpha_i}^i(\lambda)=\frac{1}{\lambda};\label{eq:limgalphai}\\
\text{C2}&: & \exists \,\,\hat c_i>0 \text{ such that }\forall
\,\, \lambda \in(0,\norm{B_i}^2] \text{ and } \forall\,\,\alpha_i>0 \text{ there holds } \abs{g_{\alpha_i}^i(\lambda)}\leq \hat c_i\alpha_i^{-1};\label{eq:cotagalphai}\\
\text{C3}&: & \exists \,\,\mu_{\scriptscriptstyle 0}^i\geq 1
\text{ such that if }\mu\in[0,\mu_{\scriptscriptstyle 0}^i]\text{
then }\lambda^\mu\abs{r_{\alpha_i}^i(\lambda)} \leq
c^i_\mu\alpha_i^\mu\hspace{0.2cm} \forall \,\,\lambda
\in(0,\norm{B_i}^2],\label{eq:icalif}
\end{eqnarray*}
where the $c^i_\mu$'s are positive constants. Let us denote now
with $\vec\alpha=(\alpha_{\s 1},\alpha_{\s 2},\ldots,\alpha_{\s
N})^T$ the ``vector-valued regularization parameter'' and with
$\vec g_{\vec\alpha}:\RE\longrightarrow \RE^N$ the function given
by $\vec g_{\vec\alpha}(\lambda)=\left(g_{ \alpha_1}^{\s
1}(\lambda),g_{ \alpha_2}^{\s 2}(\lambda),\ldots,g_{ \alpha_N}^{\s
N}(\lambda)\right)^T$ and let $\vec g_{\vec\alpha}(\vec B^*\vec
B)$ be the linear continuous self-adjoint operator on $\X^N$
defined via (\ref{eq:int-espec-vect}). Let $y\in\D(T^\dagger)$,
$y_{\s 1}^\delta,y_{\s 2}^\delta,\ldots, y_{\s N}^\delta\in\Y$ be
such that $\norm{y_{\s i}^\delta -y}\leq
\delta\,\,\forall\,\,i=1,2,...,N$ and $\vec y\,^\delta\doteq(y_{\s
1}^\delta,y_{\s 2}^\delta,\ldots, y_{\s N}^\delta)^T\in\Y^N$. We
define the regularized solution $x_{\vec\alpha}^\delta$ of problem
(\ref{eq:prob-inv}) given the observations\, $y_{\s
1}^\delta,y_{\s 2}^\delta,,\ldots, y_{\s N}^\delta$, with
regularization methods $g_{\alpha_1}^{\s 1}(\cdot),g_{
\alpha_2}^{\s 2}(\cdot),$ $\dots,g_{ \alpha_N}^{\s N}(\cdot)$, in
the Hilbert scales $\X_{\s s_1}^{\s L_1},\, \X_{\s s_2}^{\s
L_2},\ldots,\X_{\s s_N}^{\s L_{\s N}}$ induced by the operators
$L_{\s 1}, L_{\s2}, \ldots, L_{\s N}$ over $\X$, with the weights
$\eta_{\s 1},\eta_{\s 2},\ldots,\eta_{\s N}$, as
\begin{equation}\label{eq:reg esc Hilbert multiples vec}
x_{\vec\alpha} ^\delta= x_{\vec\alpha} ^\delta\left(\vec g_{\vec
\alpha},\,\vec\eta,\,\vec L,\,\vec y\,^\delta,\vec s\right)\doteq
\vec \eta{\,\s\bullet}\left(\vec L^{-\vec s} \vec g_{\vec
\alpha}(\vec B^*\vec B)\vec B^* \vec y\,^\delta\right).
\end{equation}
Suppose also that $\forall \,\,i,\, 1\leq i\leq N$, there exists
$u_i\in [0,a_i+2s_i]$ such that $x^\dag=T^\dag y \in \X_{ \s
u_i}^{\s L_i}$, i.e. $\vec x^\dag\in \X_{\s\vec u}^{\s\vec L}$,
where $\vec x^\dag\doteq(x^\dag,x^\dag,\ldots,x^\dag)^T$ and $\vec
u\doteq (u_{\s 1}, u_{\s 2},\ldots,u_{\s N})^T$. If the
vector-valued regularization parameter $\vec\alpha$ is chosen in
the form
\begin{equation}\label{eq:elecalphamultiples vect}
\vec\alpha(\delta)=\left(c_{\s 1}\,\delta^{\varepsilon_{\s
1}},c_{\s 2}\, \delta^{\varepsilon_{\s 2}},\ldots, c_{\s
N}\,\delta^{\varepsilon_{\s N}}\right)^T
\end{equation}
where $c_i>0$ and $0<\varepsilon _i<\frac{2(a_i+s_i)}{a_i},\,
1\leq i\leq N$, then:
\begin{enumerate}
\item[i)] $\norm{x_{\vec\alpha}^\delta-x^\dag}\rightarrow 0$
for $\delta\rightarrow 0^{\s +}$.
\item[ii)]  $\norm{x_{\vec\alpha}^\delta-x^\dag}= \mathcal
O(\delta^\sigma)$, where $\ds\sigma\doteq\min_{\s 1\leq i\leq
N}\sigma_i$,
$\ds\sigma_i=\min\left\{1-\frac{a_i\varepsilon_i}{2(a_i+s_i)},\frac{u_i\varepsilon_i}{2(a_i+s_i)}\right\}$.
\item[iii)] The order of convergence of the total error is optimal
when the vector regularization parameter in
(\ref{eq:elecalphamultiples vect}) is chosen such that
$\varepsilon_i=\frac{2(a_i+s_i)}{a_i+u_i}$, in which case one
obtains $\norm{x_{\vec\alpha}^\delta-x^\dag}= \mathcal
O(\delta^{\sigma^{\s *}})$, where $\ds\sigma^*=\min_{\s 1\leq
i\leq N}\frac{u_i}{a_i+u_i}$.
\item[iv)] The optimal order $\mathcal O(\delta^{\sigma^{\s *}})$
in iii) which is obtained with this vector-valued (regularization
method) $\vec g_{\vec\alpha}$, is at least as good as the optimal
order $\mathcal O(\delta ^{\sigma_\ce})$\; which is obtained with
a single observation and
 a scalar $g_\alpha(\lambda)$ (see Theorem \ref{theorem:ordenes de
convergencia multiples} \textit{iii)}).
\end{enumerate}
\end{thm}
\begin{proof}
If $\ds \{E^{\scriptscriptstyle \vec B^*\vec
B}_{\lambda}\}_{\lambda \in \mathbb{R}}$ and
$\{E^{\scriptscriptstyle B_i^*B_i}_{\lambda}\}_{\lambda \in
\mathbb{R}}$ denote the spectral families of the operators $\vec
B^*\vec B$ and $B^*_iB_i$, respectively, from the definition of
$\vec B$ and (\ref{eq:B escalas multiples}), it can be immediately
seen that $[\vec B^*\vec B\,\vec x]_i=B^*_iB_i\,x_i$ and
$[E_\lambda^{\scriptscriptstyle \vec B^*\vec B}\,\vec
x]_i=E_\lambda^{\scriptscriptstyle B_i^*B_i}\,x_i$ and therefore,
from (\ref{eq:int-espec-vect}), it follows that $\forall\,\,\vec
x=(x_{\s 1},x_{\s 2},\ldots, x_{\s N})^T\in\X^N$
$$\vec g_{\vec\alpha}(\vec B^*\vec B)\vec
x=\left( g_{ \alpha_1}^{\s 1}(B^*_{\s 1}B_{\s 1})x_{\s 1},g_{
\alpha_2}^{\s 2}(B^*_{\s 2}B_{\s 2})x_{\s 2},\ldots,g_{
\alpha_N}^{\s N}(B^*_{\s N}B_{\s N})x_{\s N}\right)^T.$$
As in Theorem \ref{theorem:ordenes de convergencia multiples}, let
$x_{i,\alpha_i} ^\delta\in \X_{\s s_i}^{\s L_i}$ be defined by
\begin{equation}\label{eq:kquita}
x_{i,\alpha_i} ^\delta =L_i^{-s_i}g^{\s i}_{\alpha_i}(B_i^*B_i)B_i^*y_i^\delta.
\end{equation}
For each index $i$, $1\leq i\leq N$, let
$\ds\sigma_i\doteq\min\left\{1-\frac{a_i\varepsilon_i}{2(a_i+s_i)},\frac{u_i\varepsilon_i}{2(a_i+s_i)}\right\}$
and $\ds\sigma\doteq\min_{\s 1\leq i\leq N}\sigma_i$. Then
\begin{align}
\norm{x_{\vec\alpha}^\delta -x^\dag}&=\norm{\vec
\eta{\,\s\bullet}\left(\vec L^{-\vec s}\vec g_{\vec\alpha}(\vec B^*\vec B)\vec B^* \vec y\,^\delta\right) -x^\dag}& \nonumber\\
&=\norm{\sum_{i=1}^N \eta_i L_i^{-s_i}g^{\s i}_{\alpha_i}(B_i^*B_i)B_i^* y_i^\delta -x^\dag}& \nonumber\\
&=\norm{\sum_{i=1}^N \eta_i x_{i,\alpha_i} ^\delta-x^\dag}& \hspace{-1cm}(\text{by (\ref{eq:kquita})})\nonumber\\
&=\norm{\sum_{i=1}^N \eta_i \left(x_{i,\alpha_i}^\delta-x^\dag\right)}& \hspace{-1cm}\left(\text{since }\sum_{i=1}^N\eta_i=1\right)&\nonumber\\
&\leq \sum_{i=1}^N \eta_i \norm{x_{i,\alpha_i}^\delta-x^\dag}&\label{eq:kquita2}\\
&\leq \sum_{i=1}^N \eta_i \,c_i\,\delta^{\sigma_i} &  \hspace{-1cm}(\text{for } \vec\alpha \text{ as in (\ref{eq:elecalphamultiples vect}), by Theorem \ref{coro:coro de ordenes de convergencia} \textit{ii)}})\nonumber\\
&\leq C\, \delta^\sigma.&\nonumber
\end{align}
This proves \textit{i)} and \textit{ii)}.

Now, if the vector-valued regularization parameter $\vec\alpha$ in
(\ref{eq:elecalphamultiples vect}) is chosen so that
$\varepsilon_i=\frac{2(a_i+s_i)}{a_i+u_i},\,\,\forall
i=1,2,\ldots,N$, then by virtue of Theorem \ref{coro:coro de
ordenes de convergencia} \textit{iii)} it follows that there exist
positive constants $c_{\s 1},c_{\s 2},\ldots,c_{\s N}$, such that
$\norm{x_{i,\alpha_i}^\delta-x^\dag}\leq
c_i\,\delta^{\frac{u_i}{a_i+u_i}},\,\,\forall i=1,2,\ldots,N$.
Then it follows from (\ref{eq:kquita2}) that
\begin{align*}
\norm{x_{\vec\alpha}^\delta-x^\dag}&\leq \sum_{i=1}^N \eta_i
c_i\,\delta^{\frac{u_i}{a_i+u_i}}\\
&\leq C\delta^{\sigma^{\s *}},
\end{align*}
where $\ds\sigma^*=\min_{\s 1\leq i\leq N}\frac{u_i}{a_i+u_i}$. It
is also clear that for $u_i$ and $a_i$ fixed, this order of
convergence is optimal and, as we can see, independent of the
choice of  $\vec s$. This proves \textit{iii)}.

Finally, to prove \textit{iv)} we must verify that
$\sigma_\ce\leq\sigma^*$, where $\sigma_\ce$ is the optimal order
in Theorem \ref{theorem:ordenes de convergencia multiples}
\textit{iii)}, that is
$$\ds\sigma_\ce=\frac{\ds\min_{\s 1\leq i \leq
N}\frac{u_i}{2(a_i+s_i)}}{\ds\min_{\s 1\leq i\leq
N}\frac{u_i}{2(a_i+s_i)}+ \ds\max_{\s 1\leq i\leq
N}\frac{a_i}{2(a_i+s_i)}}.$$
For that, observe that since $a_i,\,u_i$ and $s_i$ are all
positive, there holds
\begin{align*}
\max_{\s 1\leq i\leq N}\left(\frac{2(a_i+s_i)}{u_i}\right)\max_{\s
1\leq i\leq N} \left(\frac{a_i}{2(a_i+s_i)}\right)&\geq \max_{\s
1\leq i\leq N}\left(\frac{a_i}{u_i}\right),
\end{align*}
or equivalently
\begin{align*}
\frac{\ds\max_{\s
1\leq i\leq N} \left(\frac{a_i}{2(a_i+s_i)}\right)}{\ds\min_{\s 1\leq i\leq N}\left(\frac{u_i}{2(a_i+s_i)}\right)}&\geq
\max_{\s
1\leq i\leq N}\left(\frac{a_i}{u_i}\right),
\end{align*}
from where it follows that
\begin{align*}
\frac{1}{1+ \frac{\ds\max_{\s
1\leq i\leq N} \left(\frac{a_i}{2(a_i+s_i)}\right)}{\ds\min_{\s 1\leq i\leq N}\left(\frac{u_i}{2(a_i+s_i)}\right)}}&\leq
\frac{1}{1+\ds\max_{\s
1\leq i\leq N}\left(\frac{a_i}{u_i}\right) }\\
&=\frac{1}{\ds\max_{\s
1\leq i\leq N}\left(1+\frac{a_i}{u_i}\right)}\\
&=\frac{1}{\ds\max_{\s 1\leq i\leq
N}\left(\frac{u_i+a_i}{u_i}\right)},
\end{align*}
and therefore
\begin{align*}
\ds\sigma_\ce=\frac{\ds\min_{\s 1\leq i \leq
N}\frac{u_i}{2(a_i+s_i)}}{\ds\min_{\s 1\leq i\leq
N}\frac{u_i}{2(a_i+s_i)}+ \ds\max_{\s 1\leq i\leq
N}\frac{a_i}{2(a_i+s_i)}}\leq \ds\min_{\s
1\leq i\leq N} \left(\frac{u_i}{a_i+u_i}\right),
\end{align*}
that is $\sigma_\ce\leq\sigma^*$, as we wanted to prove.
\end{proof}
In the presence of a fixed noise level $\delta$ in the $N$
observations $y_{\s 1}^\delta,y_{\s 2}^\delta,,\ldots, y_{\s
N}^\delta$, in light of Theorem \ref{coro:coro de ordenes de
convergencia} \textit{iii)}, one should not expect that the order
of convergence $\mathcal O(\delta^{\sigma^{\s *}})=\mathcal
O(\delta^{\min\frac{u_i}{a_i+u_i}})$ in Theorem
\ref{theorem:ordenes de convergencia multiples vect} can be
improved. However, if the noise levels can be controlled, then by
appropriately doing so on those components on which it is known
that the degree of regularity of the exact solution $x^\dag$ on
the corresponding Hilbert scale (measured in terms of $u_i$) is
relatively small or the corresponding parameter of comparison of
relative regularity between the operators $T$ and $L_i^{-1}$,
measured in terms of $a_i$ (see (\ref{eq:LiT suavidad})\,), is
relatively large, then the order of convergence $\mathcal
O(\delta^{\sigma^{\s *}})$ can in fact be improved. More precisely
we have the following result.
\begin{thm}\label{theorem:ordenes de convergencia multiples vectorial
extendido} Let $\X,\,\Y,\,\X^N,\,\Y^N,\,T,\,\vec T,\,L_i,\,
u_i,\,a_i,\,1\le i\le N,\,\vec u,\,\vec s,\,\vec L,\,\vec B=\vec
T\vec L^{-\vec s}$, $ B_i=TL^{-s_i},\,\vec\alpha,\,$ $\vec
g_{\vec\alpha}, \,y\in\D(T^\dagger)\,, x^\dag=T^\dag y$, $\vec
x^\dag\in \X_{\s\vec u}^{\s\vec L}$ and $\vec\eta$, all as in
Theorem \ref{theorem:ordenes de convergencia multiples vect}. Let
$y_{\s 1}^{\delta_{\s 1}},y_{\s 2}^{\delta_{\s 2}},\ldots, y_{\s
N}^{\delta_{\s N}}\in\Y$ be such that $\norm{y_{\s i}^{\delta_i}
-y}\leq \delta_i\,\,\forall\,\,i=1,2,...,N,$
$\vec\delta=(\delta_{\s 1},\delta_{\s 2},\ldots,\delta_{\s N})^T$
and $\vec y\,^{\vec\delta}\doteq(y_{\s 1}^{\delta_{\s 1}},y_{\s
2}^{\delta_{\s 2}},\ldots, y_{\s N}^{\delta_{\s N}})^T\in\Y^N$ and
define now the regularized solution $x_{\vec\alpha}^{\vec\delta}$
of the problem $Tx=y$ as
\begin{equation}\label{eq:reg esc Hilbert multiples vec ext}
x_{\vec\alpha} ^{\vec\delta}= x_{\vec\alpha}
^{\vec\delta}\left(\vec g_{\vec \alpha},\,\vec\eta,\,\vec L,\,\vec
y\,^{\vec\delta},\vec s\right)\doteq \vec
\eta{\,\s\bullet}\left(\vec L^{-\vec s}\,\vec g_{\vec \alpha}(\vec
B^*\vec B)\vec B^* \vec y\,^{\vec\delta}\right).
\end{equation}
If $\delta_i=\delta^{p_i}$ with
\begin{equation}\label{eq:pesos ruido multiples}
p_i\geq\frac{\ds\max_{\s 1\leq k \leq N}\frac{u_k}{a_k+u_k}}{\frac{u_i}{a_i+u_i}},
\end{equation}
for every $1\leq i\leq N$, and the vector-valued regularization
parameter $\vec\alpha(\vec\delta)$ is chosen in the form
\begin{equation}\label{eq:elecalphamultiples vect ext}
\vec\alpha(\vec\delta)=\left(c_{\s 1}\delta_{\s 1}^{
\frac{2(a_1+s_1)}{a_1+u_1}}, c_{\s 2}\delta_{\s 2}^{
\frac{2(a_2+s_2)}{a_2+u_2}},\ldots,c_{\s N}\delta_{\s N}^{
\frac{2(a_N+s_N)}{a_N+u_N}} \right)^T,
\end{equation}
where $c_{\s 1},\,c_{\s 2},\ldots,c_{\s N}$ are arbitrary positive
constants, then
\begin{equation}\label{eq:thebestorder}
\norm{x_{\vec\alpha}^{\vec\delta}-x^\dag}=\mathcal O(\delta^{\hat\sigma}),
\end{equation}
where $\ds\hat\sigma\doteq\max_{\s 1\leq i\leq
N}\frac{u_i}{a_i+u_i}.$
\end{thm}
\begin{proof}
Let $\alpha_i\doteq c_{\s i}\delta_{\s i}^{
\frac{2(a_i+s_i)}{a_i+u_i}}$ and $x_{i,\alpha_i} ^{\delta_i}
=L_i^{-s_i}g^{\s i}_{\alpha_i}(B_i^*B_i)B_i^*y_i^{\delta_i}$. By
virtue of Theorem \ref{coro:coro de ordenes de convergencia}
\textit{iii)} it follows that there exist constants $k_{\s 1},
k_{\s 2},\ldots, k_{\s N}$ such that
\begin{equation}\label{eq:chuchi2}
\norm{x_{i,\alpha_i} ^{\delta_i}-x^\dag}\leq k_i\delta_i^{\frac{u_i}{a_i+u_i}}, \hspace{0.2cm}1\leq i\leq N.
\end{equation}
On the other hand, by following the same steps as in Theorem
\ref{theorem:ordenes de convergencia multiples vect}, for
$x_{\vec\alpha}^{\vec\delta}$ defined as in (\ref{eq:reg esc
Hilbert multiples vec ext}) one has that
\begin{align*}
\norm{x_{\vec\alpha}^{\vec\delta} -x^\dag}
&\leq \sum_{i=1}^N \eta_i \norm{x_{i,\alpha_i}^{\delta_i}-x^\dag}&\\
&\leq \sum_{i=1}^N \eta_i \,k_i\,\delta_i^{\frac{u_i}{a_i+u_i}} & (\text{by (\ref{eq:chuchi2})})\\
&=\sum_{i=1}^N \eta_i \,k_i\,\delta^{\frac{p_iu_i}{a_i+u_i}} & (\text{since }\delta_i=\delta^{p_i})\\
&\leq C\, \delta^{\ds\max_{\s 1\leq i\leq N} \frac{u_i}{u_i+a_i}}&(\text{by (\ref{eq:pesos ruido multiples})})\\
&=C\,\delta^{\hat\sigma}.
\end{align*}
\end{proof}
Note that in order to obtain the order of convergence in
(\ref{eq:thebestorder}) it is necessary that the noise level in
the $i^\text{th}$ component be $\delta_i=\delta^{p_i}$ with
$p_i\geq\frac{\ds\max_{\s 1\leq k \leq
N}\frac{u_k}{a_k+u_k}}{\frac{u_i}{a_i+u_i}}$ ($\geq 1\,\forall
\,\,i$). Hence, the precision in the observations must be improved
precisely in those components for which the regularity of $x^\dag$
as an element of the corresponding Hilbert scale, namely $u_i$, is
relatively small or the parameter $a_i$ is large.
\section{Conclusions}\label{sec:conclusions}
In this article several convergence results in Hilbert scales
under different source conditions are proved and orders of
convergence and optimal orders of convergence were derived. Also,
relations between those source conditions were proved. The concept
of a multiple Hilbert scale on a product space was introduced,
regularization methods on these scales were defined, first for the
case of a single observation and then for the case of multiple
observations. In the latter case, it was shown how vector-valued
regularization functions in these multiple Hilbert scales can be
used. In all cases convergence was proved and orders and optimal
orders of convergence were shown.

\bibliographystyle{amsplain}
\bibliography{ref}

\providecommand{\bysame}{\leavevmode\hbox to3em{\hrulefill}\thinspace}
\providecommand{\MR}{\relax\ifhmode\unskip\space\fi MR }
\providecommand{\MRhref}[2]{%
  \href{http://www.ams.org/mathscinet-getitem?mr=#1}{#2}
}
\providecommand{\href}[2]{#2}
\begin{thebibliography}{10}

\bibitem{ref:Acar-Vogel-1994}
R.~Acar and C.~R. Vogel, \emph{Analysis of bounded variation penalty methods
  for ill-posed problems}, Inverse Problems \textbf{10} (1994), 1217--1229.

\bibitem{refb:Dautray-Lions-1990}
R.~Dautray and J.-L. Lions, \emph{Mathematical analysis and numerical methods
  for science and technology. {V}ol. 3: Spectral theory and applications},
  Springer-Verlag, Berlin, 1990.

\bibitem{refthes:Egger-2005}
H.~Egger, \emph{Preconditioning iterative regularization methods in {H}ilbert
  scales}, Ph.D. thesis, Johannes Kepler Universität, 2005.

\bibitem{refb:Engl-Hanke-Neubauer-1996}
H.~W. Engl, M.~Hanke, and A.~Neubauer, \emph{Regularization of inverse
  problems}, Mathematics and its Applications, vol. 375, Kluwer Academic
  Publishers Group, Dordrecht, 1996.

\bibitem{refb:Friedrichs-1973}
K.~O. Friedrichs, \emph{Spectral theory of operators in {H}ilbert space},
  Springer-Verlag, New York, 1973, Applied Mathematical Sciences, Vol. 9.

\bibitem{ref:Hadamard-1902}
J.~Hadamard, \emph{Sur les problèmes aux dérivées partielles et leur
  signification physique}, Princeton University Bulletin \textbf{13} (1902),
  49--52.

\bibitem{ref:Heinz-1951}
E.~Heinz, \emph{Beitr\"age zur {S}t\"orungstheorie der {S}pektralzerlegung},
  Math. Ann. \textbf{123} (1951), 415--438.

\bibitem{ref:Krein-Petunin-1966}
S.~G. Krein and Ju.~I. Petunin, \emph{Scales of {B}anach spaces}, Uspehi Mat.
  Nauk \textbf{21} (1966), no.~2 (128), 89--168.

\bibitem{ref:Mazzieri-Spies-Temperini-2011}
G.~L. Mazzieri, R.~D. Spies, and K.~G. Temperini, \emph{Existence, uniqueness
  and stability of solutions of generalized {T}ikhonov-{P}hillips functionals},
  Preprints of the Institute for Mathematics and Its Applications, University
  of Minnesota (August 2011), no.~2375, submitted.

\bibitem{ref:Natterer-1984-AA}
F.~Natterer, \emph{Error bounds for {T}ikhonov regularization in {H}ilbert
  scales}, Applicable Anal. \textbf{18} (1984), no.~1-2, 29--37.

\bibitem{refb:Pazy-1983}
A.~Pazy, \emph{Semigroups of linear operators and applications to partial
  differential equations}, Applied Mathematical Sciences, vol.~44,
  Springer-Verlag, New York, 1983.

\bibitem{ref:Phillips-1962}
D.~L. Phillips, \emph{A technique for the numerical solution of certain
  integral equations of the first kind}, J. Assoc. Comput. Mach. \textbf{9}
  (1962), 84--97.

\bibitem{ref:Seidman-1980}
T.~I. Seidman, \emph{Nonconvergence results for the application of
  least-squares estimation to ill-posed problems}, J. Optim. Theory Appl.
  \textbf{30} (1980), no.~4, 535--547.

\bibitem{ref:Spies-Temperini-2006}
R.~D. Spies and K.~G. Temperini, \emph{Arbitrary divergence speed of the
  least-squares method in infinite-dimensional inverse ill-posed problems},
  Inverse Problems \textbf{22} (2006), no.~2, 611--626.

\bibitem{ref:Tikhonov-1963-SMD-2}
A.~N. Tikhonov, \emph{Regularization of incorrectly posed problems}, Soviet
  Math. Dokl. \textbf{4} (1963), 1624--1627.

\bibitem{ref:Tikhonov-1963-SMD-1}
\bysame, \emph{Solution of incorrectly formulated problems and the
  regularization method}, Soviet Math. Dokl. \textbf{4} (1963), 1035--1038.

\bibitem{refb:Wahba-1990}
G.~Whaba, \emph{Spline models for observational data}, SIAM, Philadelphia,
  1990.

\end{thebibliography}

\end{document}